\documentclass{amsart}
\usepackage{amssymb,latexsym, amsmath, amscd}
\usepackage{enumerate}
\usepackage{fullpage}
\usepackage{graphicx}
\usepackage{enumitem}

\makeatletter
\@namedef{subjclassname@2010}{%
\textup{2010} Mathematics Subject Classification}
\makeatother

\newtheorem{thm}{Theorem}
\newtheorem{cor}[thm]{Corollary}

\newtheorem{prop}[thm]{Proposition}

\theoremstyle{definition}

\newtheorem{rem}[thm]{Remark}
\newtheorem{exa}[thm]{Example}

\numberwithin{equation}{section}

\newcommand*{\vcenteredhbox}[1]{\begingroup
\setbox0=\hbox{#1}\parbox{\wd0}{\box0}\endgroup}

\makeatletter
\def\set@curr@file#1{%
  \begingroup
    \escapechar\m@ne
    \xdef\@curr@file{\expandafter\string\csname #1\endcsname}%
  \endgroup
}
\def\quote@name#1{"\quote@@name#1\@gobble""}
\def\quote@@name#1"{#1\quote@@name}
\def\unquote@name#1{\quote@@name#1\@gobble"}
\makeatother

\def\eps{\varepsilon}
\def\fig#1{\vcenteredhbox{\includegraphics[height=0.8in]{#1}}}
\def\ffig#1#2{\vcenteredhbox{\includegraphics[height=#1]{#2}}}

\def\E{E}
\def\eps{\varepsilon}
\def\Z{\mathbb Z}

\def\A{\mathcal A}

\def\R{\mathbb R}

\def\D{\mathcal{D}}
\def\clasp{clasp\ }

\begin{document}

\title{Encoding knots by clasp diagrams}

\author[J. Mostovoy]{J. Mostovoy}
\address{Departamento de Matem\'aticas, CINVESTAV-IPN, Av.\ IPN 2508, Col.\ San Pedro Zacatenco, Ciudad de M\'exico, C.P.\ 07360\\ M\'exico}
\email{jacob@math.cinvestav.mx}
\author[M. Polyak]{M. Polyak}
\address{Department of Mathematics, Technion, Haifa 32000, Israel}
\email{polyak@math.technion.ac.il}


\begin{abstract} We introduce a new combinatorial method to encode knots and links with applications to knot invariants.
Clasp diagrams defined in this paper are combinatorial blueprints for building knot diagrams out of full twists on two strings rather than out of crossings. We describe an equivalence relation on clasp diagrams which produces the isotopy classes of knots  as equivalence classes. This equivalence relation is generated by local moves similar to the Reidemeister moves.

Clasp diagrams produce particularly simple Seifert surfaces for knots and lead to an explicit formula for the Alexander-Conway polynomial. They are also well-suited for the study of the Vassiliev invariants; we show that any such invariant can be obtained via subdiagram count in the clasp diagrams.
\end{abstract}

\thanks{The second author was partially supported by the ISF grant 1794/14 and the FORDECYT grant 265667.}

\maketitle

\section{Introduction}

A knot in $\R^3$ is most easily represented by a generic planar projection, with over- and undercrossings indicated. Such a projection can be encoded combinatorially by a Gauss diagram which is a set of distinct points on a circle grouped into several ordered pairs, graphically represented by arrows, each arrow endowed with a sign. Knot projections and Gauss diagrams are convenient for calculating knot invariants; however, they have a drawback: there is no general way to smooth a crossing on a knot projection without creating a link with more than one component, or to remove an arrow from a Gauss diagram without creating a virtual knot. 

One may think of a Gauss diagram as ``generated'' by the crossings of a knot. Similarly, any braid can be written as a product of ``crossings''; that is, elementary braids. However, in the study of pure braids one has an (often) better option: any pure braid is a product of ``full twists''. Full twists are well-suited for various problems, such as proving that the pure braid group is residually nilpotent; this fact may also be stated as ``finite-type invariants distinguish braids''. Here we propose a variation on the concept of a Gauss diagram which uses full twists, rather than crossings; these are the clasp diagrams in the title of this paper. We prove that each knot can be represented by a clasp diagram and identify a finite set of moves such that two diagrams represent the same knot if and only if they can be transformed into each other by a sequence of these moves.

The values of certain knot invariants can be calculated by counting the subdiagrams of the Gauss diagram of a knot \cite{PV}. 
In particular, the value of any given Vassiliev invariant on an arbitary knot can be calculated in this way; this result, due to Goussarov, \cite{GPV} is far from trivial. We show that Goussarov's theorem has its counterpart for clasp diagrams; in contrast with the Gauss diagram case, the proof is completely straightforward. 

Clasp diagrams also turn out to be an efficient tool to encode the Seifert surface of a knot; as a consequence of this, the  Alexander polynomial of a knot can be easily calculated from its clasp diagram. Clasp diagrams can also be useful in the computations of other polynomial invariants defined by skein relations.

The idea of replacing crossings in a knot diagram by full twists has been exploited before; in particular, it is implicit in Goussarov's notion of $I$-modification \cite{G}.  
The closest definition to that of a clasp diagram was given by Hugelmeyer in \cite{Hugelmeyer}: his signed chord diagrams are what we call ``descending clasp diagrams'' and his motivation (namely, understanding the Vassiliev invariants) is close to ours. However, the set of signed chord diagrams turns out to be too small; as a result, the equivalence relation on these diagrams is quite complicated. We consider a larger set of diagrams than in \cite{Hugelmeyer} and obtain considerably simpler equivalence relations, which are similar to the Reidemeister moves on knots. The relation between the clasp diagrams and Hugelmeyer's signed chord diagrams is roughly similar to that between all pure braids and combed pure braids; in particular, the minimal number of chords required to encode a given knot will typically be smaller for a clasp diagram than for a signed chord diagram. We prove Hugelmeyer's main theorem as part of our results.

The paper is organized as follows. In the next section we define the clasp diagrams, explain how they represent knots and string links and describe, in Subsection~\ref{subsec:moves}, a finite set of moves which do not change the isotopy class of the knot represented by a diagram. We also state the main theorems of the paper. Theorem~\ref{thm:representation} affirms that each knot can be represented by a clasp diagram and Theorem~\ref{thm:descending} says that this diagram can be chosen so as to be of a special form. Theorem~\ref{thm:moves} says that two diagrams represent the same knot if and only if they are related by a sequence of the moves defined in Subsection~\ref{subsec:moves}. In Section~\ref{section:alex} we describe the Seifert surface coming from a clasp diagram and show how this leads to the computation of the Alexander polynomial. Section~\ref{section:vassiliev} is dedicated to the Vassiliev invariants; in particular, we show how to construct the universal invariant of order $n$.  Sections~\ref{section:shortcircuit} and \ref{section:proof} are dedicated to the proofs of the main results. 

One necessary note about the terminology. There are many ways to think about parametrized knots in the three-dimensional space and we will use three of them. A usual ``compact'' or ``round'' knot is a smooth embedding $S^1\to \R^3$. A ``long'' knot is a smooth embedding $\R\to \R^3$ which coincides with the embedding $t\mapsto (0,0,-t)$ outside of some compact subset of $\R$. According to a weaker definition, a long knot is a smooth embedding $\R\to \R^3$ whose tangent vector, outside of some compact subset of $\R$, coincides with $(0,0,-1)$. The equivalence classes of knots under smooth homotopy are the same for all three definitions and can be canonically identified; we will use this fact as obvious and give no further explanations. We will use mostly round knots for pictures, long knots coinciding with the $z$-axis at infinity in the context of string links and long knots with the fixed tangent vector at infinity whenever we speak of braid closures. 

\section{Clasp diagrams}
\subsection{Definitions}\label{subsec:clasps}

We will use the following terminology. A \emph{chord diagram} is a finite set of intervals smoothly embedded into an oriented two-dimensional disk whose boundary carries a finite set of marked points. The endpoints of the intervals, all distinct, lie on the boundary of the disk away from the marked points. The set of marked points, if non-empty, has a distinguished point.

The embedded intervals are called \emph{chords} and the boundary of the disk without the marked points is called the \emph{skeleton} of the diagram. A chord diagram whose set of marked points is empty is called \emph{compact}; its skeleton is a circle. A chord diagram with at least one marked point is referred to as a \emph{string link} chord diagram; its skeleton consists of several copies of $\R$. A diagram with exactly one marked point is called \emph{long}. Chord diagrams are considered up to smooth homotopies of the chords and smooth dispacements of the marked points that keep all the endpoints and the marked points distinct.

 A \emph{\clasp diagram} is a chord diagram whose chords are ordered and equipped with a sign. This order on the chords will be called \emph{height} and shown in pictures by means of over- and undercrossings, see Figure~\ref{fig:clasp_diagrams}.

\begin{figure}[htb]
%\vspace{1.2in}
\includegraphics[height=0.9in]{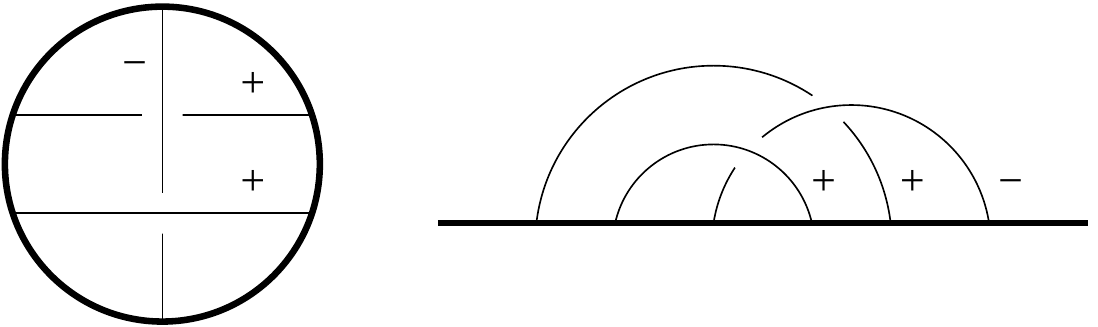}
\caption{Compact and long clasp diagrams}
 \label{fig:clasp_diagrams}
\end{figure}

We will omit signs of the chords from the pictures of clasp diagrams when these are not relevant. When drawing parts of clasp diagrams we will always assume that the chords shown on the picture have consecutive heights and that the marked points lie on the omitted part of the circle. The skeleta of long clasp diagrams will be often drawn as straight lines.

The chords of a long clasp diagram with $n$ chords have another natural order on them; namely, the order of their left endpoints along the skeleton. We will enumerate the chords in this order by $1,2,\dots,n$.

A long clasp diagram can be encoded by an $n\times n$ matrix. 
For a pair $i$, $j$ of intersecting chords with $i$ passing over $j$ define $l_{ij}=1$ if 
$i<j$ and  $l_{ij}=-1$ if $i>j$. Set all remaining elements of $L$ to be $0$. 
Let also $\E$ be the diagonal matrix encoding the signs of all chords, that is, with 
$\eps_{ii}=\pm 1$ being the sign of the chord $i$. 
The matrices $L$ and $\E$ (or the matrix $L+\E$) determine the diagram uniquely unless its chords can be separated into two nonempty subsets so that the chords from different subsets are disjoint.

The \emph{mirror} $D^*$ of a clasp diagram $D$ is the diagram obtained from $D$ by inverting the height and all the signs simultaneously. %The operation of taking the mirror commutes with the connected sum:
%$$(A\# B)^* = A^*\# B^*.$$

\subsection{Realizing knots by clasp diagrams}\label{subsec:realization}
When the chords of a compact clasp diagram are replaced by linked pairs of strings as shown below, the clasp diagram gives rise to a knot: 
$$\includegraphics[height=1.3in]{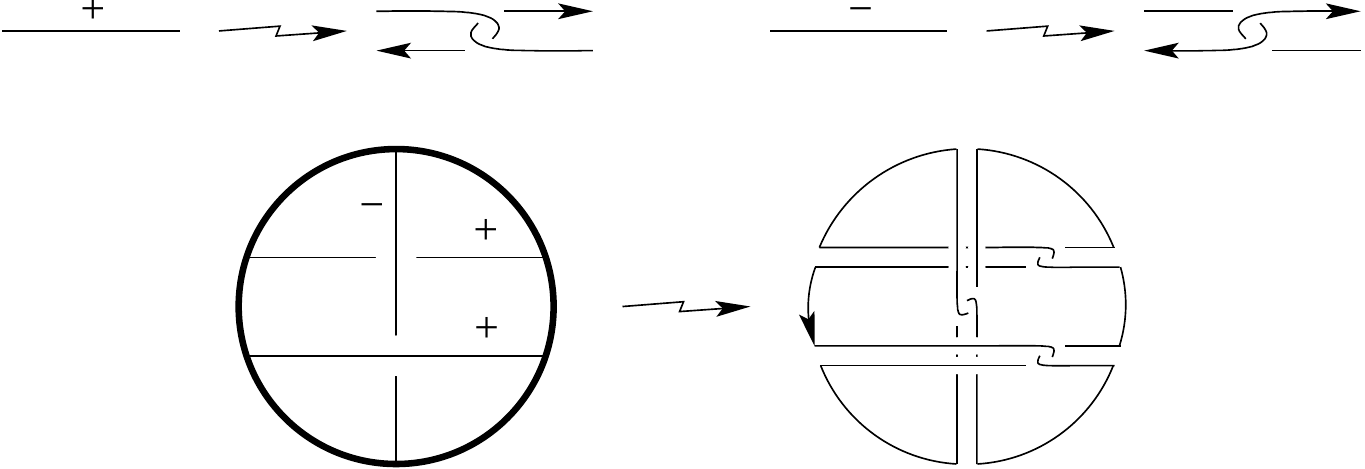}$$
The $\pm$ sign of a chord determines whether the pair of strings has linking number 
$\pm1$. 
We say that a knot obtained in this fashion from a clasp diagram \emph{realizes} it, and that the corresponding diagram \emph{represents} the knot. Long knots are represented by long clasp diagrams while compact knots realize compact diagrams.
Examples of knots realizing some clasp diagrams with two and three chords are shown in Figure \ref{fig:knot_table}.

Similarly, a string link clasp diagram gives rise to a string link:
$$\includegraphics[height=0.6in]{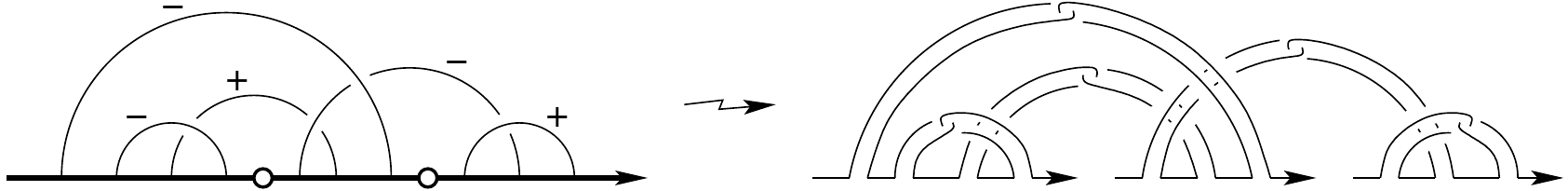}$$
Here, in order to make sense out of the right-hand side of the figure, it should be assumed that the $k$th component of the string link lies strictly below the $k+1$st component everywhere apart from small neighbourhoods of the crossings, for $k=1,2$.

Note that the mirror operation on clasp diagrams represents the mirror image operation on knots. 

\begin{figure}[htb]
%\vspace{1in}
\includegraphics[height=1.5in]{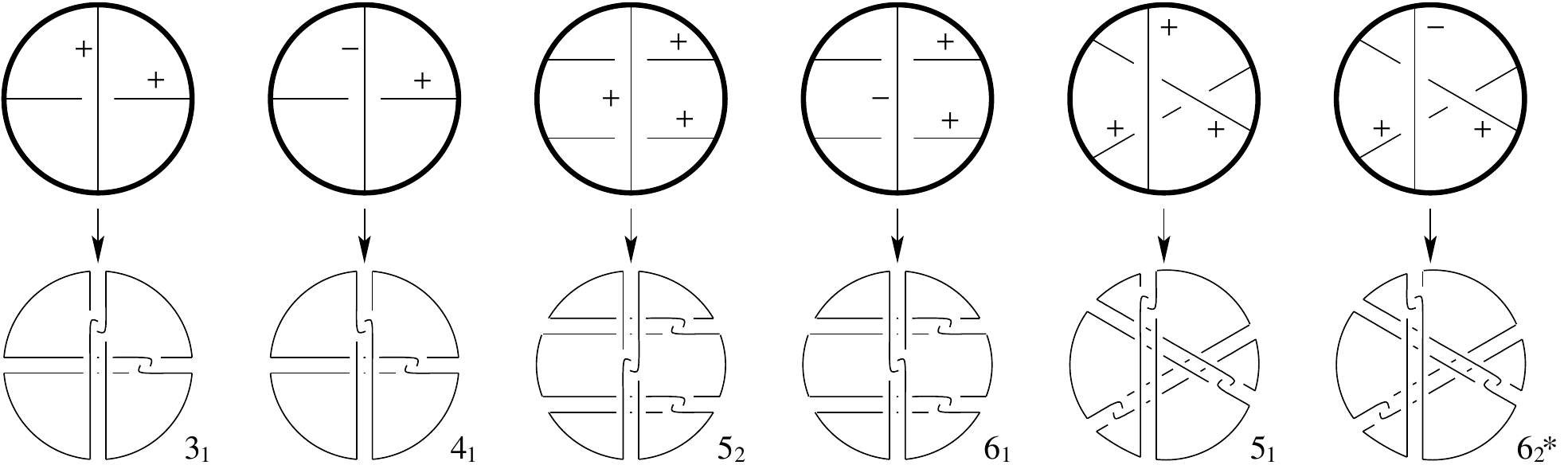}
\caption{Knots corresponding to clasp diagrams}
\label{fig:knot_table}
\end{figure}

\begin{thm}\label{thm:representation}
Each isotopy class of knots can be represented by a clasp diagram. 
\end{thm}
We will prove this statement for long knots and diagrams; as a consequence, it also holds for compact knots and diagrams. 

A clasp diagram representing a knot may be chosen so as to be of a special form. 
A long clasp diagram is \emph{descending}, if its chords are numbered in the order inverse to the order by height. In other words, for each pair of chords, the left end of the upper chord lies to the left of the left end of the lower chord.

\begin{thm}\label{thm:descending}
Each isotopy class of knots can be represented by a descending clasp diagram.
\end{thm}
In fact, this statement is also true for string links and the proof, which we omit, is essentially the same.

\subsection{Moves on clasp diagrams}\label{subsec:moves}
Different \clasp diagrams may produce equivalent knots. In particular, the following transformations do not alter the isotopy type of the corresponding knot:
\begin{itemize}[align=parleft, labelsep=0.9cm, leftmargin=*]
\item[(A)\,  ] an exchange of the order of two non-intersecting chords with consecutive heights;
\item[(B)\,  ] a cyclic shift of the order (heights) of the chords;
\item[($\mathrm{C}_1$)] erasing an isolated chord:
$$\includegraphics[height=0.5in]{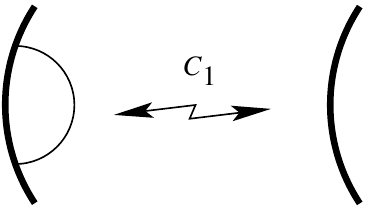}$$
\item[($\mathrm{C}_2$)] erasing a pair of parallel chords with consecutive heights and opposite signs:
$$\includegraphics[height=0.5in]{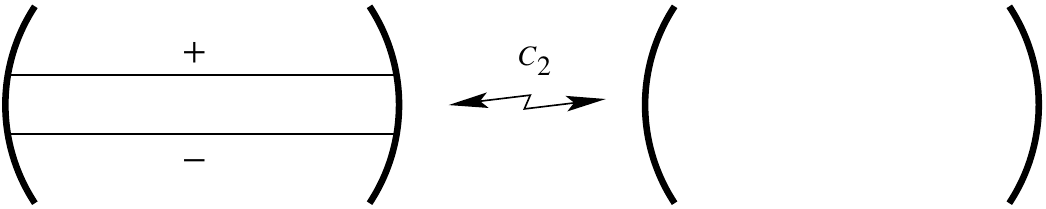}$$
\item[($\mathrm{C}_4$)] the four-clasp move:
$$\includegraphics[height=1.0in]{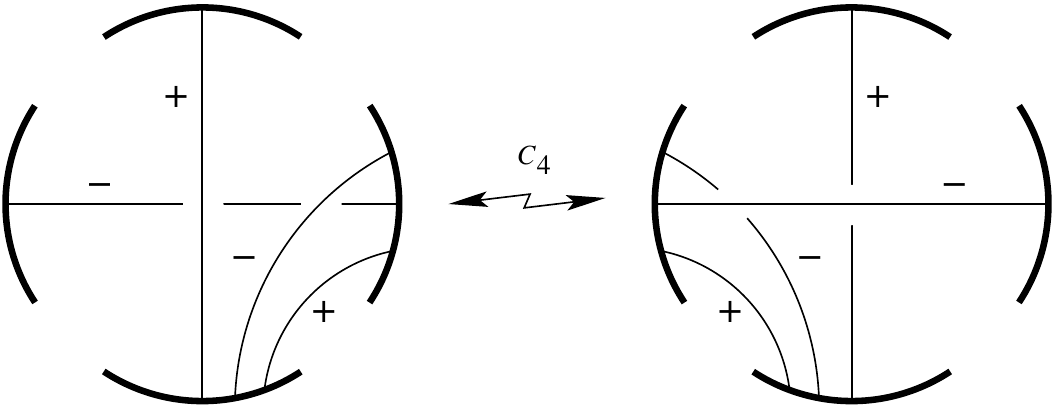}$$
\end{itemize}

Indeed, the move A produces exactly the same knot projection. The isotopy corresponding to the move B consists in rotating the highest pair of linked strings around the knot:
\smallskip

$$\includegraphics[height=1in]{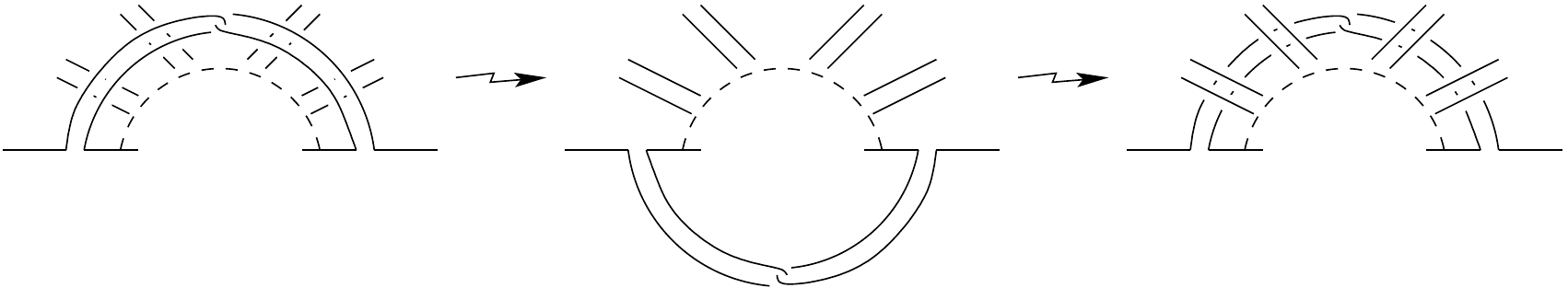}$$
The realization of the move $\mathrm{C}_1$ follows from the first Reidemeister move:
\smallskip

$$\includegraphics[height=0.3in]{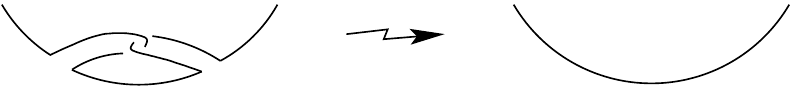}$$
Similarly, the realization of  $\mathrm{C}_2$ follows from the second Reidemeister move:
\smallskip

$$\includegraphics[height=0.5in]{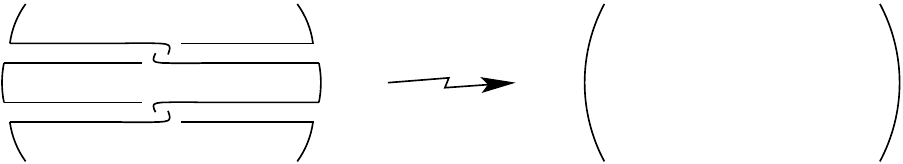}$$
The four-clasp move gets translated into the following  equivalence of knot projections:

\smallskip

$$\includegraphics[height=0.9in]{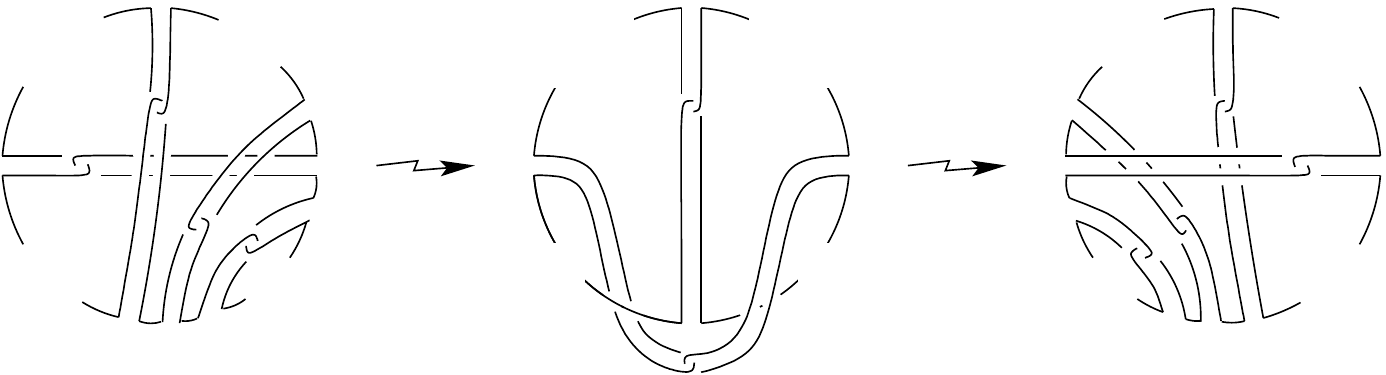}$$

\begin{thm}\label{thm:moves}
Any two clasp diagrams that give rise to isotopic knots are related by  a finite sequence of moves 
{\rm A, B}, $\mathrm{C}_1$, $\mathrm{C}_2$ and $\mathrm{C}_4$. 
\end{thm}
Here, again, on can consider either long or compact diagrams. This statement also holds for string links, with the same proof.

\subsection{Diagrams enhanced by Y-graphs}
Let us briefly mention a generalization of the notion of a clasp diagram that will not be used elsewhere in this paper. 

In the Goussarov's version \cite{G} of the Goussarov-Habiro theory, the insertion of a clasp into a knot diagram is a result of an \emph{I-modification}. The central role in \cite{G} belongs to \emph{Y-modifications}, and these can also be reflected in clasp diagrams.

Consider clasp diagrams where, in addition to chords, one can have Y-graphs connecting triples of distinct points on the skeleton of the diagram. If a chord represents a clasp, a Y-graph stands for a copy of the Borromean rings, with each component connected to a point on the skeleton. Essentially, a Y-graph is a shorthand for the configuration of four consecutive chords as in Figure~\ref{fig:borro}.
\begin{figure}[htb]
%\vspace{1.0in}
\includegraphics[height=0.8in]{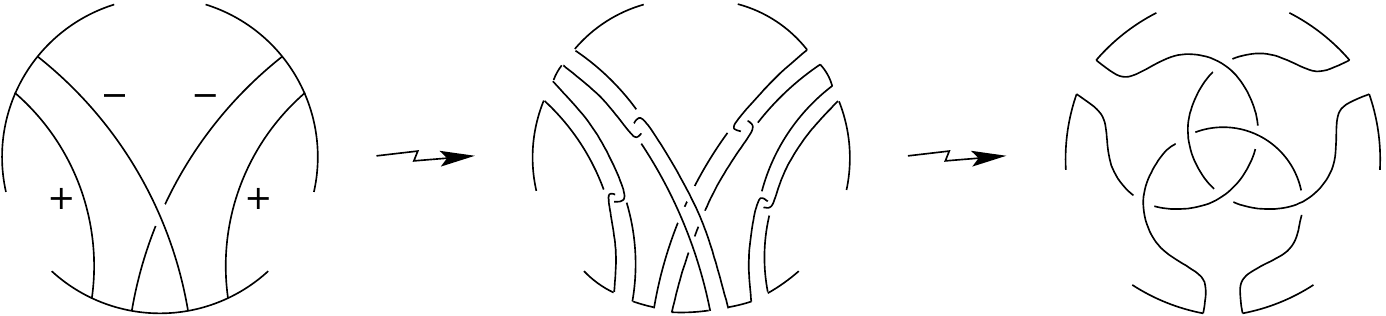}
\caption{A configuration of four chords which corresponds to a Y-graph}
\label{fig:borro}
\end{figure}

In Section~\ref{section:shortcircuit} we will see that chords in a clasp diagram correspond to generators (or their inverses) in the braid groups. In these terms, a Y-graph corresponds to a commutator of two generators.

\section{Seifert surfaces and the Alexander polynomial}\label{section:alex}
Let $D$ be a clasp diagram with $n$ chords. Consider a knot diagram realizing $D$. 
It is easy to see that it has a natural genus $n$ Seifert surface with $n$ flat bands glued to 
the disk along the chords of $D$ and $n$ short $\pm 1$-twisted bands 
(resulting from clasps) glued across them near the middle, see Figure~\ref{fig:Seifert}. 

\begin{figure}[htb]
%\vspace{1.0in}
\includegraphics[height=0.8in]{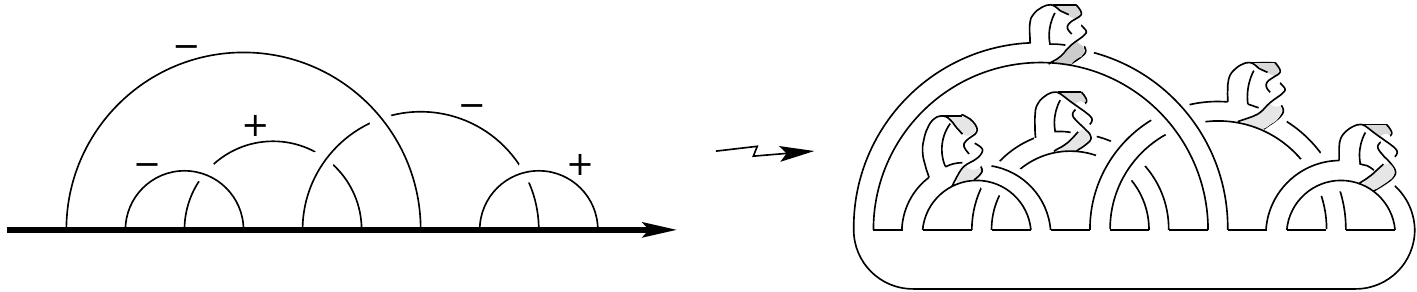}
\caption{A Seifert surface corresponding to a clasp diagram}
\label{fig:Seifert}
\end{figure}

Let $L$ and $E$ be the linking and the sign matrices of $D$, respectively, see Section~\ref{subsec:clasps}. 
Define an $n\times n$ matrix $S_D$ with coefficients in $\Z[t^{-1},t]$ as follows. 
For each pair $i$, $j$ of linked chords of $D$ with $i$ passing over $j$ define $s_{ij}=t-1$, 
$s_{ji}=t^{-1}-1$ if $i<j$;  and  $s_{ij}=1-t$, $s_{ji}=1-t^{-1}$ if $i>j$. 
Set $s_{ii}=-e_{ii}$ for all $i$ and let all remaining elements of $S$ to be $0$.

\begin{prop}
\label{prop:Alexander}
The Alexander polynomial $\Delta_K (t)$ of the knot $K$ corresponding to $D$ satisfies
$$\Delta_K(t)=\det S_D\,.$$ 
\end{prop}

\begin{proof}
After choosing the basis in the first homology of this surface as the cores of all bands 
(with cores of long bands with the counterclockwise orientation taken first, in the order 
of their left endpoints) we obtain the corresponding  $2n\times 2n$ Seifert matrix 
$$V=\left(\begin{array}{c|c}L & 0\\ \hline I & -E\end{array}\right)$$ 
made of four $n\times n$ blocks, with the diagonal blocks being $L$ and $-E$, the lower off-diagonal 
block being the identity $n\times n$ matrix, and the upper off-diagonal block being zero.
%The Alexander polynomial of $K$ is $\Delta_K(t)=\det(t^{1/2}V-t^{-1/2}V^T)$. 
Recall that the determinant of a block matrix 
$$\left(\begin{array}{c|c}A & B\\ \hline C & D\end{array}\right)$$ 
with an invertible block $D$ equals $\det(A-BD^{-1}C)\cdot\det D$; thus 
$$\Delta_K(t)=\det(t^{1/2}V -t^{-1/2}V^T)=(t^{1/2} -t^{-1/2})^n 
\det \left( t^{1/2}L -t^{-1/2}L^T-(t^{1/2} -t^{-1/2})^{-1}E\right)\cdot\det(-E) \,. $$ 
Since $\Delta_K(t)$ is defined up to a sign and a multiplication by $t$, 
we can disregard the sign $\det(-E)$. 
The theorem follows from a straightforward identification of the matrix 
$(t^{1/2} -t^{-1/2}) \left( t^{1/2}L -t^{-1/2}L^T\right)-E$ with $S_D$.
\end{proof}

\begin{exa}
For the clasp diagram of the $6_1$ knot in Figure \ref{fig:knot_table} 
put a base point just before the negative chord. Then we have 
$$
L=\left(\begin{array}{c c c} 0 & 1 & 1 \\ 0 & 0 & 0 \\ 0 & 0 & 0 \end{array}\right)\ ,\quad 
E=\left(\begin{array}{c c c} -1 & 0 & 0 \\ 0 & 1 & 0 \\ 0 & 0 & 1 \end{array}\right)\ ,\quad
S_D=\left(\begin{array}{c c c} 1 & t-1 & t-1 \\ t^{-1}-1 & -1 & 0 \\ t^{-1}-1 & 0 & -1 \end{array}\right).$$ 
Thus the Alexander polynomial of $6_1$ is $\Delta_{6_1}(t)=\det S_D=5-2t^{-1}-2t$.
\end{exa}

\begin{rem}
Another way to calculate the Alexander-Conway (or, more generally, the HOMFLYPT) polynomial from $D$ is 
by using the skein relation. Applying the skein relation to each clasp, we either remove it (by a crossing change), 
or change the clasped ribbon of $K$ to a standard straight ribbon (by smoothing). 
Thus the HOMFLYPT polynomial of $K$ can be directly calculated from its values on links constructed by adding 
ribbons along different subsets of chords of $D$.   
This leads to new formulae for the finite type invariants obtained from the HOMFLYPT polynomial.  
\end{rem}

\section{Clasp diagrams and Vassiliev invariants} \label{section:vassiliev} 

Clasp diagrams provide a particularly neat approach to the Vassiliev invariants. The Vassiliev invariants are precisely those that can be calculated as the numbers of subdiagrams (taken with certain weights) of a clasp diagram of a knot. The constructions of this section follow the arguments of \cite{GPV}; however, the results that we present here are considerably more satisfactory. As usual in the theory of the finite type invariants, the term ``singular knot with $n$ double points'' will be used to denote a linear combination of $2^n$ knots obtained from a curve with $n$ transversal self-intersections by applying the Vassiliev skein relations $n$ times. By a Vassiliev invariant of order $n$ we will mean an invariant that vanishes on all singular knots with more than $n$ double points (some authors use in this situation the terminology ``Vassiliev invariant of order \emph{at most} $n$'').  For the basics of the theory of Vassiliev invariants, see \cite{CDM}.

\subsection{The Goussarov Theorem}

Denote by $\D$ the set of all clasp diagrams and let $\Z\D$ be the free abelian group spanned by $\D$. Consider the map
$I: \Z\D\to\Z\D$ which
sends a diagram to the sum of all of its subdiagrams. For instance,
$$\includegraphics[height=0.6in]{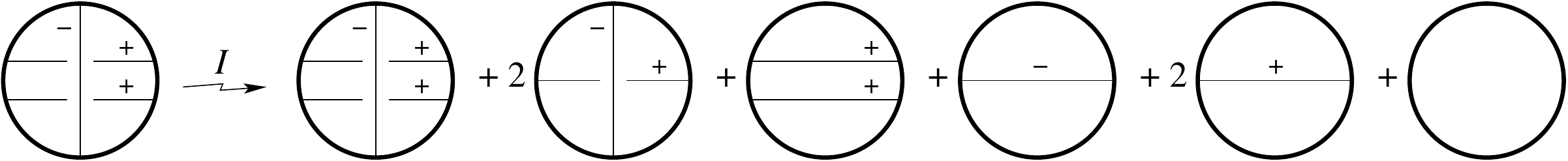}$$

Differences $D-D'$ of pairs of diagrams related by the  moves {\rm A, B}, $\mathrm{C}_1$, $\mathrm{C}_2$ and $\mathrm{C}_4$ span the subgroup $R\subset \Z\D$; the quotient $\Z\D/R$ is the free abelian group $\Z\mathcal{K}$ generated by the isotopy classes of knots. Denote by $\A$ the quotient $\Z\mathcal{D}/I(R)$; the map $I$ induces an isomorphism $$\Z\mathcal{K}\simeq \A.$$
Write $\A_n$ for the quotient of $\A$ by all clasp diagrams with more than $n$ chords and $I_n$ for the corresponding map $\Z\mathcal{K}\to \A_n$.

\begin{thm}\label{thm:completeVI}
Let $M$ be an abelian group and $n$ a non-negative integer. Given a homomorphism $w:\A_n\to M$, the composition $w\circ I_n$ is a Vassiliev invariant of order $n$. All Vassiliev invariants of order $n$ are obtained in this way.
\end{thm}

There is another way to state this result, in the spirit of the Goussarov Theorem 
\cite{GPV}. For a pair of  clasp diagrams $A$, $B$ denote by 
$\left< A,B\,\right>$
the number of subdiagrams of $B$ isomorphic to $A$. Write $K_B$
for the knot corresponding to the clasp diagram $B$. Then, we have the
following
\begin{cor}\label{thm:Goussarov}
For each integer-valued Vassiliev invariant $v$ of order $n$, there is
a integer function on clasp diagrams $A\mapsto v_A$ which vanishes on
the diagrams with more than $n$ chords and such that
$$v(K_B) = \sum_{A\in\mathcal{D}} v_A \left< A,B\,\right>.$$
Conversely, any knot invariant which has an expression of this type is a Vassiliev invariant of order $n$.
\end{cor}
Note that $I_n$ itself is a Vassiliev invariant of order $n$, universal in the sense that any other invariant of order $n$ factors through it. Corollary~\ref{thm:Goussarov} applied to $I_n$ gives
$$I_n (K_B) = \sum_{A\in\mathcal{D}}  [A]\cdot \left< A,B\,\right>,$$
where $[A]$ is the class of the diagram $A$ in $\A_n$.

\medskip

The hard part of the Goussarov Theorem for Gauss diagrams
consists in keeping track of the realizability of the diagrams. Here
this obstacle does not exist.

\begin{proof}[Proof of Theorem~\ref{thm:completeVI}]
Singular knots can be 
represented by clasp diagrams which have \emph{special chords}.  Each special chord, drawn below with a dashed line, corresponds to  
a clasp on which one of the two crossings is replaced by a double point:
$$\includegraphics[height=0.28in]{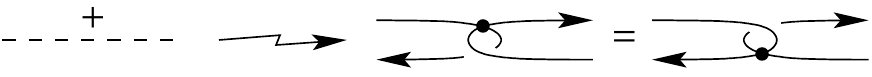}$$
$$\includegraphics[height=0.28in]{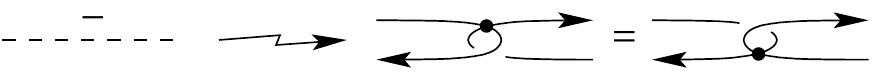}$$
Any singular knot can be represented in this way by a diagram with special chords (see Proposition~\ref{prop:singknots}).
In fact, the space of singular knots with $n$ double points is spanned by the singular knots represented by the diagrams with at least $n$ chords, \emph{all} of them special. 
Indeed, consider a clasp diagram $D$ whose special chords form a subdiagram $D''\subseteq D$. Write $|D|$ for the total number of chords of a diagram $D$ and $|D|^-$ for the number of its negative chords. 
 Then, the singular knot $K_{D}$ realizing $D$ can be written as
$$K_{D} = \sum_{D''\subseteq D'\subseteq D} (-1)^{|D|^--|D'|^-} K_{\overline{D'}},$$
where $\overline{D'}$ is obtained from ${D'}$ by making all of its chords special.

Now, observe that the map $I$ is invertible; for a clasp diagram $D$ we have
$$I^{-1}(D) = \sum_{D'\subseteq D} (-1)^{|D|-|D'|} D'.$$
For a clasp 
diagram $D$ with $k$ chords, the linear combination of
knots $$\sum_{D'\subseteq D} (-1)^{|D|-|D'|} K_{D'}$$ is, up to a sign,
a singular knot with $k$ double points, whose diagram can be obtained
from $D$ by making all its chords special.  Therefore, the map $I^{-1}$ identifies the subspace of diagrams with more than $n$ chords in $\Z\D$ with the subspace of diagrams that represent singular knots with $n$ double points.

As a consequence, the map $I_n$ vanishes on all singular knots with more than $n$
double points and $w\circ I_n$ is a Vassiliev invariant. On the other hand, consider 
an arbitrary Vassiliev invariant $v$ of order $n$ as a function on clasp diagrams. 
Then, $v\cdot I^{-1}$ is a homomorphism $\A\to M$ which vanishes on the diagrams with
more than $n$ chords and, therefore, descends to $\A_n$.
\end{proof}

\begin{exa}
Extend $\left<A,B\,\right>$ to linear combinations of clasp diagrams by linearity.
We have the following formulae for the first two Vassiliev knot invariants:
$$
v_2(K_B)=\left<\ffig{0.4in}{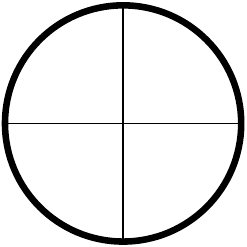}\ , B\right>\ , \qquad
v_3(K_B)=\left<\ffig{0.4in}{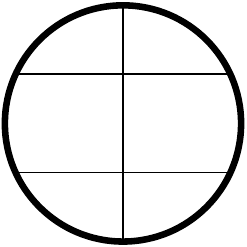}\ +2\ \ffig{0.4in}{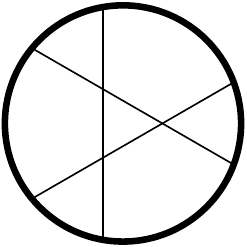}\ +\
\ffig{0.4in}{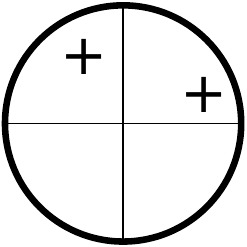}\ -\ \ffig{0.4in}{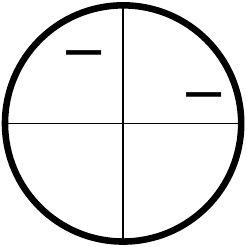}\ , B\right>
$$
Here the heights and some signs are omitted from the pictures. This means that one should take the sum over all possible decorations of these chord diagrams  with heights and signs (where applicable) of the chords. Each diagram then should be taken with the coefficient equal to the product of the omitted signs. These formulae should be compared with those of \cite{PV, GPV}.

For string links with more than one component we also have simple expressions. The invariant $v_{12}$ of degree two of a two-component string link $K_B$ represented by a clasp diagram $B$ has the form
$$v_{12}(K_B)=\left<\ffig{0.3in}{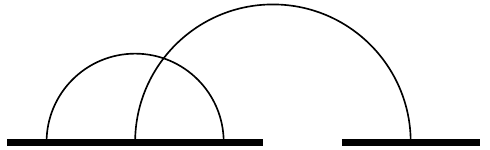}\ +\ \ffig{0.3in}{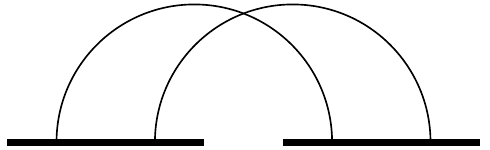}\ +\ \ffig{0.3in}{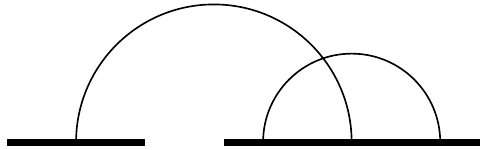} , B\right>$$
Also, up to a normalization (some combination of linking numbers) and depending on the conventions, the triple linking number 
$\mu_{12,3}$  of a three-component string link $K_B$ can be obtained as
$$\mu_{12,3}(K)=\left<\ffig{0.3in}{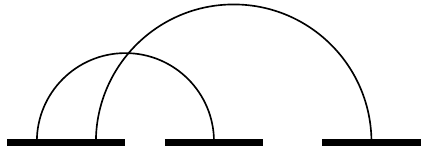}\ +\ \ffig{0.25in}{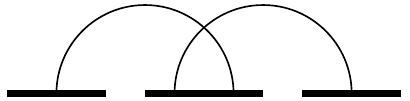}\ +\ \ffig{0.3in}{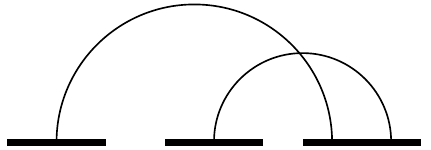} , B\right>$$

\end{exa}

\subsection{Relations in the $\A_n$} It may be instructive to describe explicitly the image of the moves  ${\rm A} - \mathrm{C}_4$ under the map $I$. 

The moves ${\rm A}$ and ${\rm B}$ do not change under $I$:  a cyclic shift of the heights of the chords of a diagram, as well as an interchange of the heights of two disjoint consecutive chords, produces the same element of $\A$. The moves $\mathrm{C}_1$ and $\mathrm{C}_2$ become
\begin{equation}
\tag{$\mathrm{C}_1'$}\ffig{0.5in}{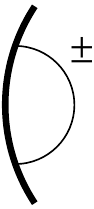}\ =0
\end{equation}
and
\begin{equation}
\tag{$\mathrm{C}_2'$}\ffig{0.5in}{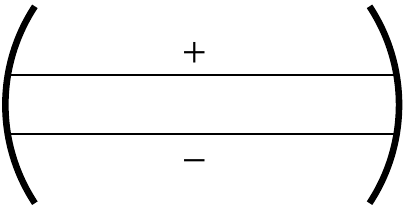}\ + \ \ffig{0.5in}{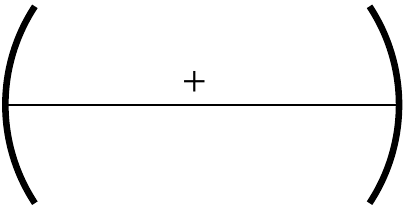}\ + \ \ffig{0.5in}{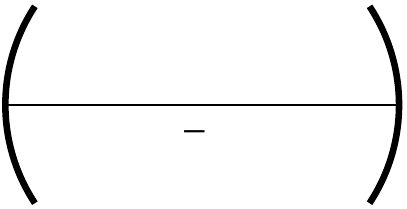}\ =0.
\end{equation}
The image of the move $\mathrm{C}_4$ under $I$ is somewhat complicated. However, in the proof of Theorem~\ref{thm:moves} we will see that that the moves $\mathrm{A} - \mathrm{C}_4$ have a number of useful consequences. One of them is the move that we call $\mathrm{D}_1$. Its mirror image $\mathrm{D}_1^*$ under $I$ transforms into
\begin{equation}
\tag{$\mathrm{D}'$}\ffig{0.7in}{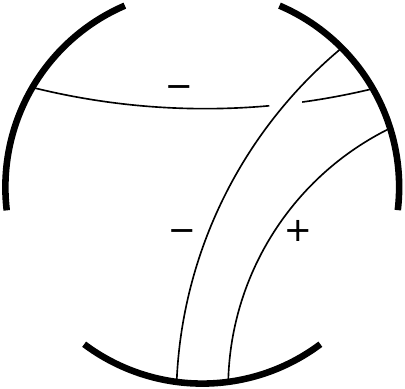}\ + \ \ffig{0.7in}{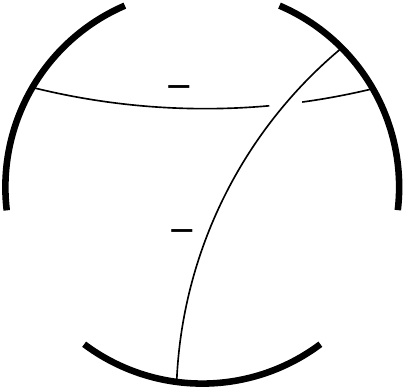}\ + \ \ffig{0.7in}{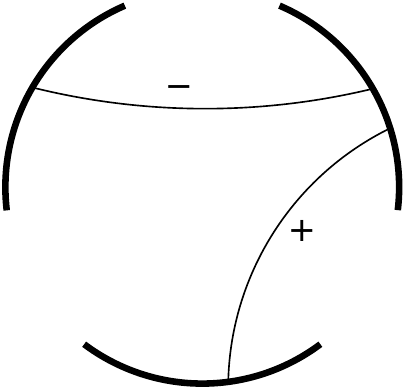}\ =\
\ffig{0.7in}{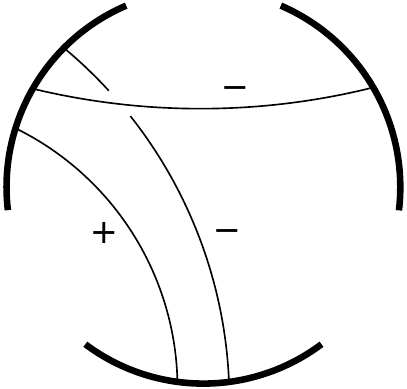}\ + \ \ffig{0.7in}{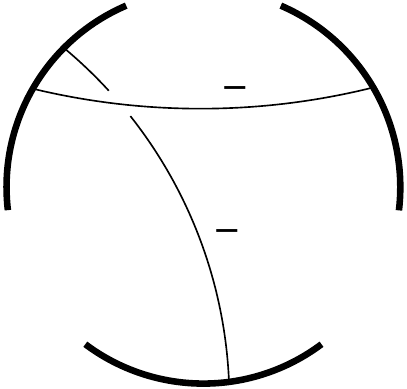}\ + \ \ffig{0.7in}{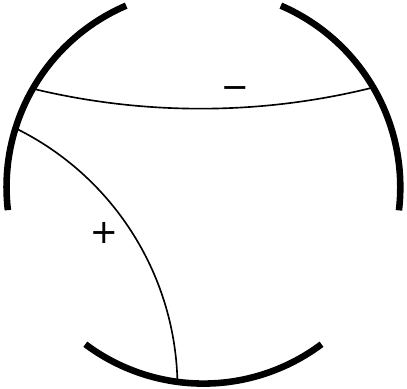}\ .
\end{equation}
Using this relation together with $\mathrm{C}_1'$ and $\mathrm{C}_2'$ we can simplify the image of $\mathrm{C}_4$ 
 so as to get 
\begin{multline}
\tag{$\mathrm{C}_4'$} \ffig{0.85in}{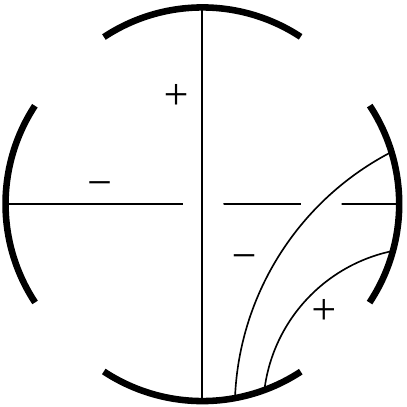}\ + \ \ffig{0.85in}{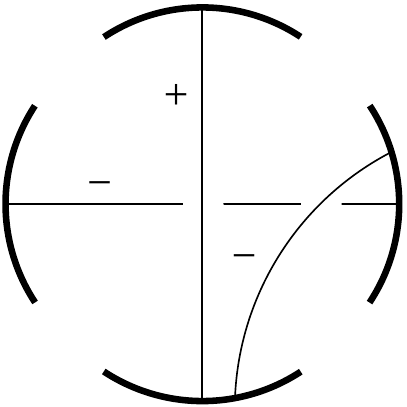}\ + \ \ffig{0.85in}{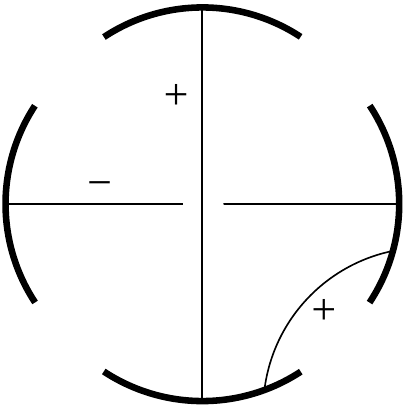}\ + \ \ffig{0.85in}{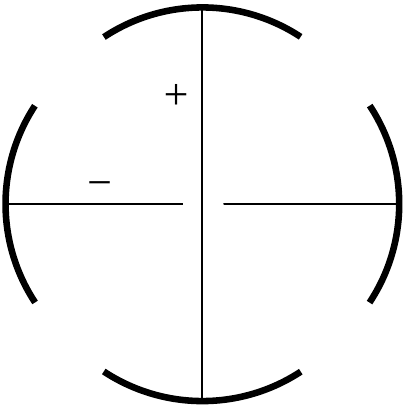}\ =\\
\ffig{0.85in}{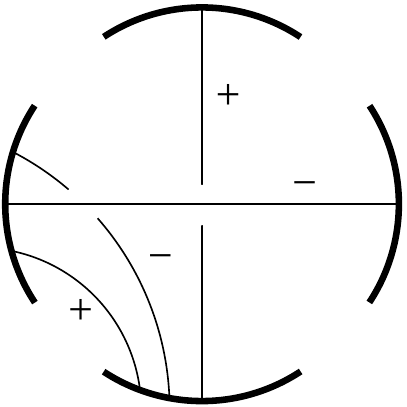}\ + \ \ffig{0.85in}{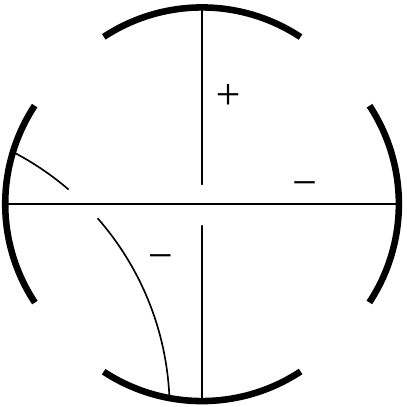}\ + \ \ffig{0.85in}{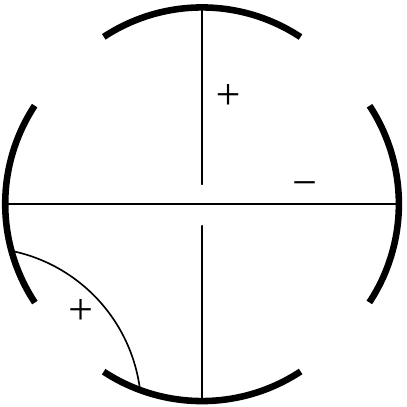}\ + \ \ffig{0.85in}{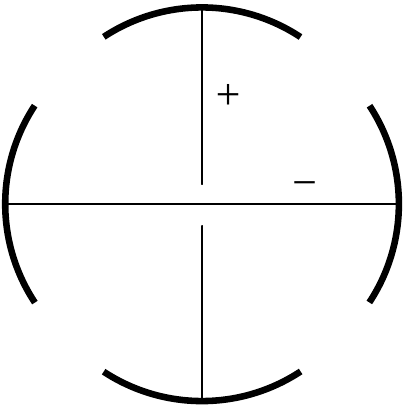}\ .
\end{multline}

This is a complete set of relations in $\A$. One obtains the set of relations for the abelian group $\A_n$, dual to the Vassiliev invariants of order $n$, by setting, in addition, all the diagrams with more than $n$ chords to be equal to zero. 

The kernel of the natural map  $\A_n\to\A_{n-1}$ is dual to the space of Vassiliev invariants of order $n$ modulo those of order $n-1$. By a fundamental result of Kontsevich it is isomorphic, up to torsion, to the abelian group generated by the usual chord diagrams modulo the 1T and 4T relations. The above relations in $\A_n$ clearly illustrate this fact. 

Indeed, consider the relations $\mathrm{C}_2'$ and $\mathrm{D}'$ involving diagrams with $n$ and $n+1$ chords and  $\mathrm{C}_4'$ involving diagrams with $n$, $n+1$ and $n+2$ chords. Since the diagrams with more than $n$ chords vanish in $\A_n$, these relations become particularly simple on diagrams with precisely $n$ chords: $\mathrm{C}_2'$ allows to eliminate the signs of the chords by changing the sign of the whole diagram; $\mathrm{C}_4'$ shows that the heights of the chords can be freely interchanged and $\mathrm{D}'$ translates into the usual 4T relation.

\begin{exa}
It is not hard to verify that $\A_3$ is a free abelian group of rank three. Denote by $H, X$ and $O$ the
equivalence classes of the following diagrams:  
$$ \ffig{0.4in}{v3_H} \quad \ffig{0.4in}{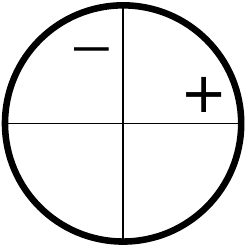}\quad \ffig{0.4in}{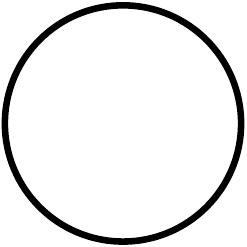}$$
Then, for any knot $K$ we have
$$I_3(K) = v_3(K) \cdot H - v_2(K) \cdot X + O.$$
\end{exa}

We stress that, while the connection of the chord diagrams with the Vassiliev invariants rests on the existence of the Kontsevich integral, the abelian groups $\A_n$  are dual to the Vassiliev invariants by definition. In particular, computations in $\A_n$ may be useful for computing the numbers of torsion-valued Vassiliev invariants.

\section{Knots as closures of pure braids}\label{section:shortcircuit} 

\subsection{Presentations of the pure braid groups}
A pure braid on $m$ strands is a homotopy class of paths in the configuration space of $m$ distinct points in $\R^2$. As the basepoint, we choose a configuration in which the $m$ points are collinear and lie on the $x$-axis. Let us label the points in this configuration by the integers from 1 to $m$ in the increasing order from left to right;  then, the strands of each braid are also numbered from 1 to $m$. 

The group $P_m$ is generated by the braids $A_{i,j}$,  with $1\leq i< j\leq m$, whose all strands, apart  from the strands $i$ and $j$, are vertical, while the strands $i$ and $j$ twist around each other once behind the other strands. The product of two braids is obtained by placing the first factor on top of the second factor; see Figure \ref{fig:pure_braids}.

\begin{figure}[htb]
%\vspace{1.0in}
\includegraphics[height=1in]{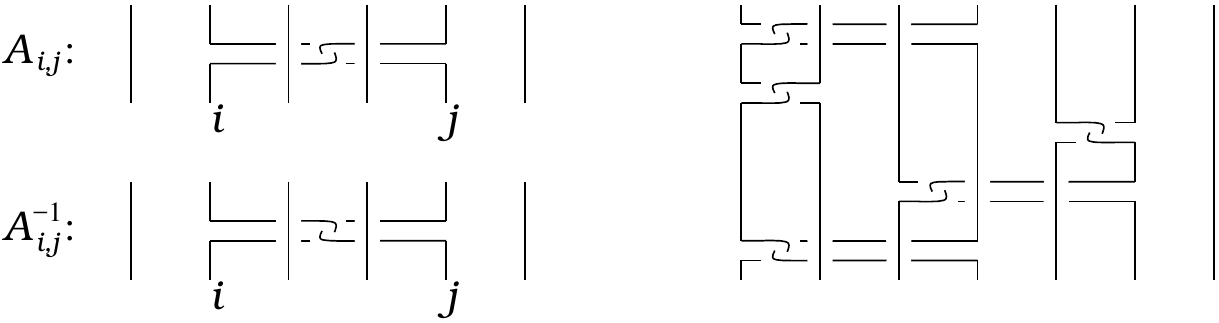}
\caption{Generators $A_{i,j}$, $A_{i,j}^{-1}$ and the braid $w=A_{1,4} A_{1,2} A^{-1}_{5,6} A_{3,6} A^{-1}_{1,4}$}
\label{fig:pure_braids}
\end{figure}

We will make use of two different presentations for $P_m$ with the same set of generators $\{A_{i,j}\}$.

The \emph{Artin presentation} \cite{Artin, ArtinE} has the relations
\begin{equation}\label{rel-artin1}
A_{r,s}^{-1}A_{i,j}A_{r,s}  = 
\begin{cases}
A_{i,j} 
\hfill \text{if}\ i<j<r<s & \text{or}\  i<r<s<j\quad  (\text{A}1)\\
  A_{i,s} A_{i,j}A_{i,s}^{-1}
& \text{if}\ i<j=r<s \quad\  (\text{A}2)\\
 (A_{i,s}A_{i,r}) A_{i,j} (A_{i,s}A_{i,r})^{-1} 
& \text{if} \ i<r<j=s \quad\  (\text{A}3)\\
 (A_{i,s}A_{i,r}A_{i,s}^{-1}A_{i,r}^{-1}) A_{i,j} 
(A_{i,s}A_{i,r}A_{i,s}^{-1}A_{i,r}^{-1})^{-1}
& \text{if}\ i<r<j<s \quad\  (\text{A}4)
\end{cases}
\end{equation}
The following relations are consequences of (A1)-(A4):
\begin{equation}\label{rel-artin2}
A_{r,s}A_{i,j}A_{r,s}^{-1}  = 
\begin{cases}
A_{i,j} 
\hfill \text{if}\ i<j<r<s & \text{or}\  i<r<s<j  \\
  (A_{i,s}A_{i,r})^{-1} A_{i,j} (A_{i,s}A_{i,r}) 
& \text{if}\ i<j=r<s \\
 A_{i,r}^{-1} A_{i,j}A_{i,r}
& \text{if} \ i<r<j=s \\
 (A_{i,r}^{-1}A_{i,s}^{-1}A_{i,r}A_{i,s})^{-1} A_{i,j} 
(A_{i,r}^{-1}A_{i,s}^{-1}A_{i,r}A_{i,s})
& \text{if}\ i<r<j<s 
\end{cases}
\end{equation}

The second presentation that will be of use is the \emph{Margalit-McCammond presentation\footnote{in fact, this is only one 
of the family of the presentations considered in \cite{McMargalit}.}} \cite{McMargalit}; it has the relations
\begin{equation}\label{rel-mcmargalit}
\begin{array}{lr}
 [A_{i,j},A_{r,s}] =1 & \text{if}\ i<j<r<s \ \text{or}\  i<r<s<j\quad  (\text{MM}1)\\[2pt]
 A_{i,r}A_{r,s}A_{i,s}=A_{r,s}A_{i,s}A_{i,r}=A_{i,s}A_{i,r}A_{r,s} & \text{if \ } i<r<s\quad  (\text{MM}2)\\[2pt]
 \left[ A_{i,j}, A_{i,r}A_{r,s}A_{i,r}^{-1}\right] =1 & \text{if}\ i<r<j<s\quad  (\text{MM}3)
 
\end{array}
\end{equation}

It is not hard to obtain both presentations from each other. 
Indeed, (MM1) is the same thing as (A1), (MM3) follows from (A4) and (A2), the first equality of (MM2) coincides with (A2) 
and the second equality is obtained from (A2) and (A3). Obtaining the Artin relations from the Margalit-McCammond relations 
is equally easy.

A pure braid written as a reduced word $w$ in the $A_{i,j}$ is said to be in the \emph{combed form} if $A_{i,j}$ appearing to the left of $A_{i'j'}$ in $w$ implies $i\leq i'$. Each pure braid has a unique combed form. Indeed, the relations (\ref{rel-artin1}) and (\ref{rel-artin2}) show how to interchange $A_{i,j}^{\pm 1}$ with the $A_{r,s}^{\pm 1}$ for $i<r$. 
The right-hand side of each of the relations in (\ref{rel-artin1}) and (\ref{rel-artin2}) only involves generators with the same 
first index $i$ and, hence,  any word in the $A_{i,j}$ after a finite number of applications of (\ref{rel-artin1}) and (\ref{rel-artin2}) can be transformed into a combed form.  
The uniqueness of the combed form is a consequence of the iterated semidirect product decomposition
$$P_m = F_{m-1}\rtimes P_{m-1} = F_{m-1}\rtimes F_{m-2}\rtimes\ldots\rtimes F_2\rtimes \Z,$$
where the free group $F_{m-k}$ is generated by the $A_{k,j}$ with $j>k$. 

\subsection{The short-circuit closure}\label{ssc}

A pure braid on an odd number of strands can be closed up to form a long knot (the \emph{short-circuit closure} of the braid):
$$
\includegraphics[height=1in]{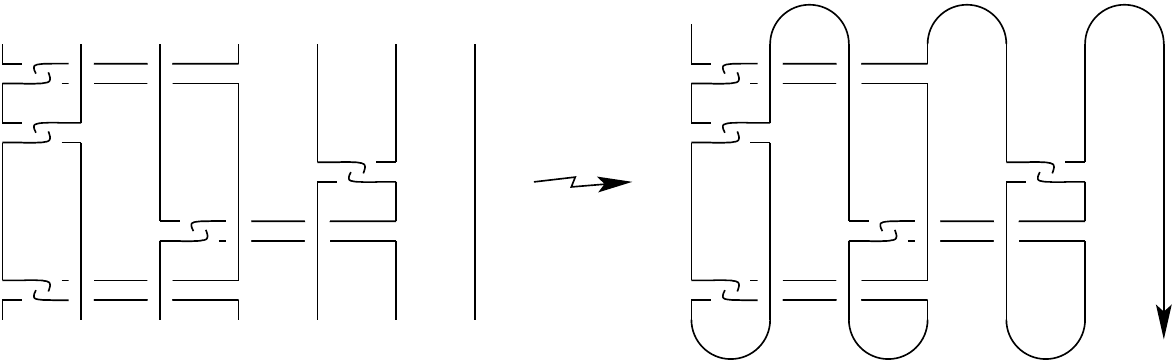}
$$

The short-circuit closure is a map from $P_{2m-1}$, with $m>0$, to the set of isotopy classes of knots. It is compatible with the inclusion maps $P_{2m-1}\to P_{2m+1}$ which add two unbraided non-interacting strands to the right; indeed, adding these two strands results in simply adding a ``bump'' to the knot. Therefore, one can speak of the short-circuit map as defined on the direct limit $P_{\infty}$ of the pure braid groups. The image of is map is the whole set of isotopy classes of knots; two braids produce the same knot if and only if they are related by a sequence of moves of two types, which we call the \emph{short-circuit moves}. The short-circuit moves of the first type exchange a braid $a$ with a braid of the form $A_{2r-1,2r}^{\pm1}\cdot  a$ or with a braid of the form $a\cdot  A_{2r,2r+1}^{\pm1}$ for some $r>0$:
$$\includegraphics[width=2.5in]{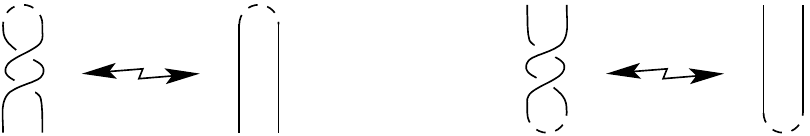}$$
The short-circuit moves of the second type exchange $a$ with one of the following:
\begin{itemize} 
\item  $(A_{r,2s+1} A_{r, 2s})^{\eps} \cdot a$  with $0<r< 2s$, 
\item $(A_{2s+1, r} A_{2s, r})^{\eps} \cdot a$ with $0< 2s +1< r$,
\item  $a\cdot (A_{r, 2s} A_{r, 2s-1})^{\eps}$ with $0<r< 2s-1$,
\item $a\cdot (A_{2s,r}A_{2s-1,r})^{\eps}$ with $0< 2s < r$,  
\end{itemize}
where $\eps=\pm1$, as illustrated below:
$$\includegraphics[width=5.4in]{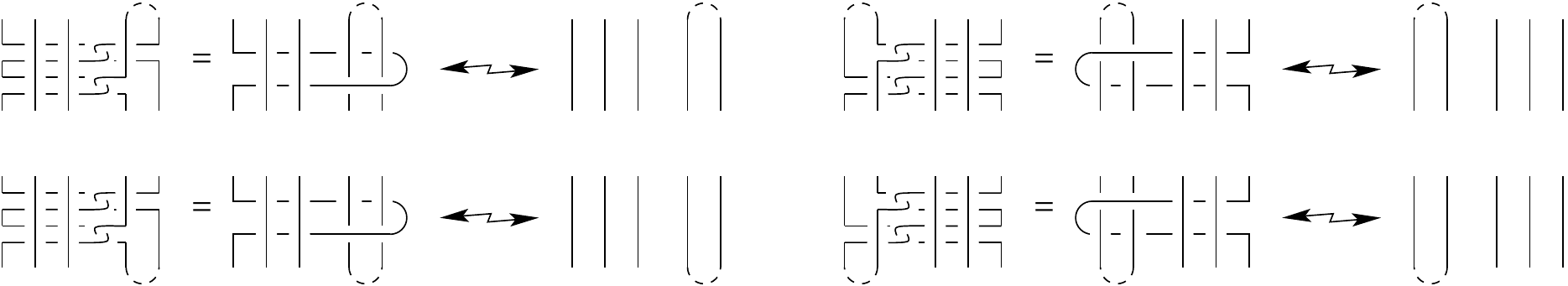}$$
See \cite{MS} or \cite{CDM} for the details of the proof.

In what follows, we will refer to the limit $P_{\infty}$ of the pure braid groups simply as \emph{the} pure braid group.

\subsection{Clasp diagrams for a short-circuit closure of a braid}\label{section:diagramsviaclosure}

Consider a word $w$ in the generators $A_{i,j}$ with indices of opposite parity. We can encode it graphically 
by $2m-1$ evenly spaced vertical strands with ends on the same heights, together with $n$ horizontal chords, 
all on different levels, each equipped with a $\pm$ sign, connecting pairs of different strands. 
The chord connecting the strands $i$ and $j$ and labelled with $\pm1$ corresponds to $A_{i,j}^{\mp1}$ 
(note the opposite sign!); we draw it ``behind'' all the vertical strands it crosses.
The corresponding word $w$ is read off this diagram from top to bottom. 
Then, one can speak of its short-circuit closure, which is a long clasp diagram. 
Namely, tilt the vertical strands, the odd ones to the left and the even ones to the right until their ends meet. 
Think of the resulting picture as three-dimensional and look at it from above. What one sees is a long clasp diagram, 
see Figure~\ref{fig:braid2clasp}.

\begin{figure}[htb]
%\vspace{1.0in}
\includegraphics[width=5in]{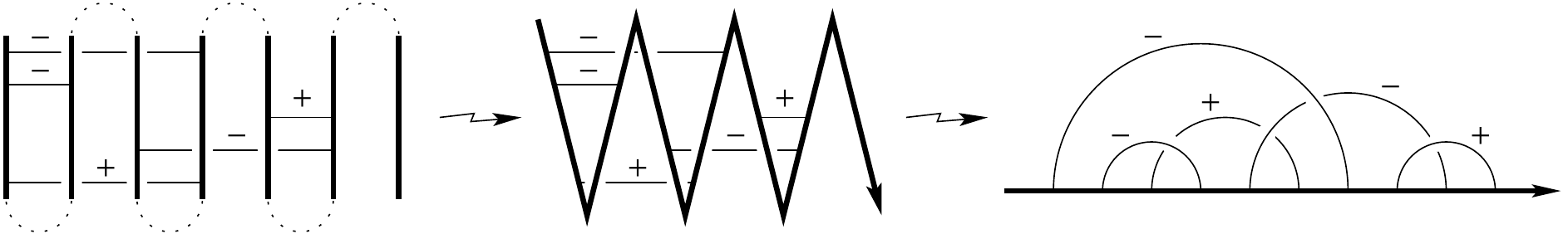}
\caption{Converting a word $w=A_{1,4} A_{1,2} A^{-1}_{5,6} A_{3,6} A^{-1}_{1,4}$ into a clasp diagram}
\label{fig:braid2clasp}
\end{figure}

It is straightforward to see that the short-circuit closure of a word in the generators $A_{i,j}$ with indices of 
opposite parity is a clasp diagram representing the short-circuit closure of the corresponding braid:

$$
\includegraphics[width=5in]{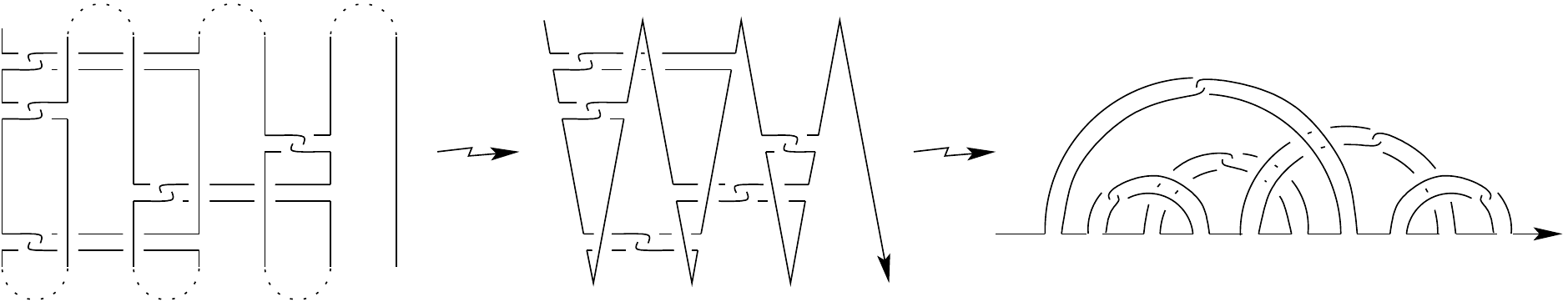}
$$ 

Moreover, if this braid is in combed form and only contains generators with the odd first index, the diagram is descending. 

\subsection{Proof of Theorems~\ref{thm:representation} and \ref{thm:descending}}

In what follows, we shall often abuse the terminology 
and use the term ``braids'' for the words in the $A_{i,j}$ rather than their equivalence classes. 
This should not lead to confusion.

From the discussion of the previous subsection we see that it is sufficient to establish that each knot is a short-circuit 
closure of a braid which is in combed form and only involves generators $A_{i,j}$ with $i$ odd and $j$ even. 
Each generator in such a braid should be of the form $A_{2r-1, 2s}$ for some positive $r\leq s$. 
Let us call all other generators (namely, the $A_{i,j}$ with $i$ even or $j$ odd) ``inadmissible''.
Consider a combed braid $a$ written as a word in the $A_{i,j}$ that includes precisely $k$ inadmissible generators 
or their inverses. 
We shall transform it, without changing its short-circuit closure, into a combed braid  that involves $k-1$ inadmissible 
generators.

Denote by $\tau_k :P_\infty\to P_\infty$ the endomorphism that adds two vertical strands between the 
$k$-th and the $(k+1)$-st strands, in front of all other strands. 
It sends $A_{i,j}$ to $A_{i,j}$ if $j\leq k$, to $A_{i,j+2}$ if $i\leq k < j$, and to  $A_{i+2,j+2}$ 
when $k < i$. Note that $\tau_k(A_{i,j})$ is inadmissible if and only if $A_{i,j}$ is.

We have $a = b A_{i,j}^{\eps} c$ where the braid $b$ involves no inadmissible generators, $A_{i,j}$ is inadmissible, 
$\eps=\pm 1$  and $c$ involves $k-1$ inadmissible generators. There are six possibilities for this according to whether 
$i$ and $j$ are even or odd and whether $\eps$ is positive or negative. In each of these cases we can transform $a$ into 
another braid that has the same short-circuit closure and involves one inadmissible generator less:

$$
\begin{array}{lcl}
 b A_{2r, 2s+1}^{-1}c &\quad\mapsto\quad& (\tau_{2r-1}\circ\tau_{2s+1})(b) \cdot A_{2r+1, 2s+4}^{-1}  \cdot (\tau_{2r}\circ\tau_{2s})(c)\\[2pt]
 b A_{2r, 2s+1}c &\quad\mapsto\quad&  (\tau_{2r-1}\circ\tau_{2s+1})(b)  \cdot A_{2r+1, 2s+4} \cdot  (\tau_{2r}\circ\tau_{2s})(c) \\[2pt]
 b A_{2r, 2s}^{-1}c &\quad\mapsto\quad &\tau_{2r-1}(b) \cdot A_{2r+1, 2s+2} \cdot \tau_{2r}(c),\\[2pt]
 b A_{2r+1, 2s+1}c &\quad\mapsto\quad& \tau_{2s+1}(b) \cdot   A_{2r+1, 2s+2}^{-1} \cdot   \tau_{2s}(c),  \\[2pt]
 b A_{2r, 2s}c &\quad\mapsto\quad&(\tau_{2r-1}\circ\tau_{2s-1}\circ\tau_{2s-1})(b) \cdot A_{2r+1, 2s+4}^{-1}A_{2s+3, 2s+6}^{-1} \cdot  (\tau_{2r}\circ\tau_{2s}\circ\tau_{2s})(c),\\[2pt]
 b A_{2r+1, 2s+1}^{-1}c &\quad\mapsto\quad& (\tau_{2s+1}\circ\tau_{2s+1})(b) \cdot  A_{2r+1, 2s+2} A_{2s+1, 2s+4} \cdot  (\tau_{2s}\circ\tau_{2s})(c)\\[2pt]
  \end{array}
$$

These transformations are illustrated on Figure~\ref{fig:guatever}, with $b=c=1$. 
The first four of them carry combed braids into combed braids. 
The last two transformations, in principle, do not; however, in both of them the generator 
that breaks the order commutes with all other generators that appear below it in the braid. 
Therefore, applying these transformations repeatedly, we arrive to a combed braid without 
inadmissible generators which represents the same knot.

\begin{figure}[htb] 
%\vspace{1.2in}
\includegraphics[height=1.6in]{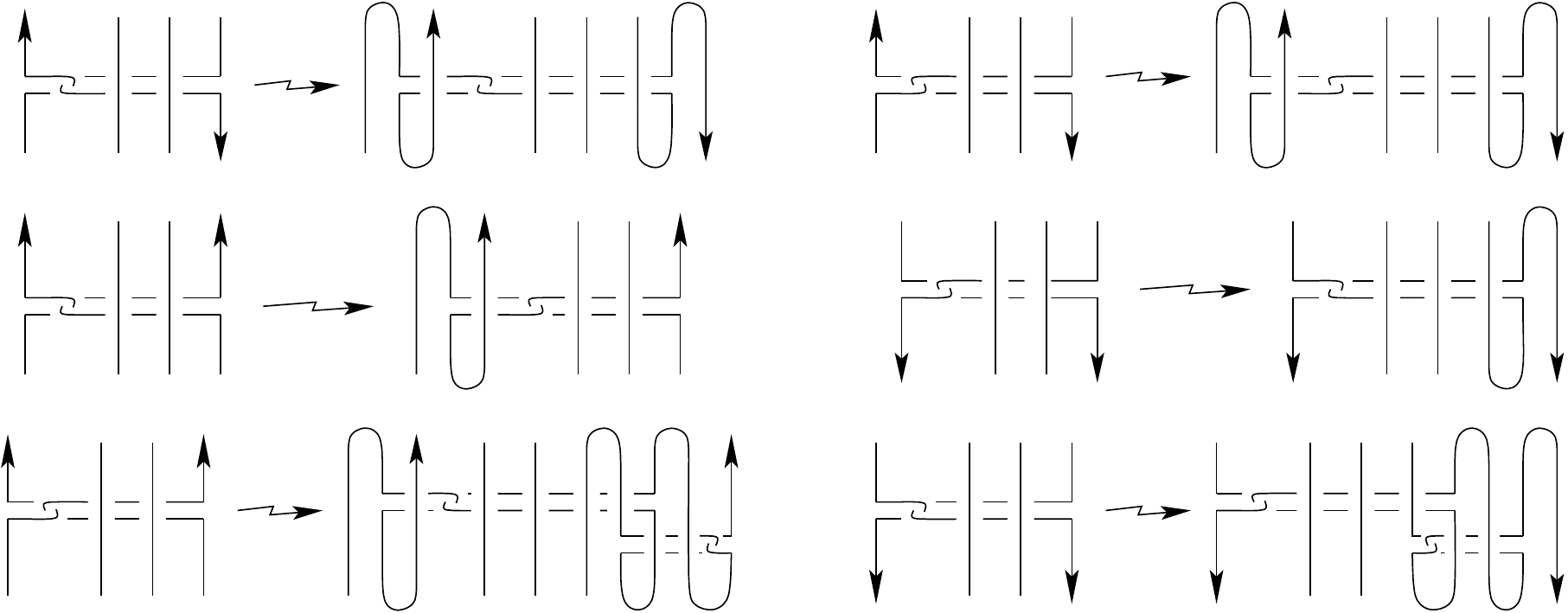}
\caption{Elimination of inadmissible generators}
\label{fig:guatever}
\end{figure}

The arguments of this proof are valid word-for-word for singular knots and clasp diagrams with special chords, defined in the proof of Theorem~\ref{thm:completeVI}. Just as a chord of a clasp diagram with the sign $\eps$ represents the generator (or its inverse) $A_{2r+1, 2s}^{-\eps}$, a special chord with the sign $\eps$ represents $\eps (A_{2r+1, 2s}^{-\eps}-1)$, where $1$ is understood as the trivial braid. In particular, we have

\begin{prop}\label{prop:singknots}
Any singular knot is a realization of a clasp diagram with special chords.
\end{prop}

\section{Proof of Theorem~\ref{thm:moves}}\label{section:proof}

Each clasp diagram can be thought of as the closure of a pure braid as in Section~\ref{section:diagramsviaclosure}. 
Two braids represent the same knot under the short-circuit closure if and only if they are related by a finite sequence 
of moves of two kinds: the relations in the pure braid group and the short-circuit moves defined in Section~\ref{ssc}. 

A braid is converted into a clasp diagram by replacing each generator $A_{i,j}$ or its inverse $A_{i,j}^{-1}$ by one or two chords, according to the parity of $i$ and $j$ as shown below:
$$
%\vspace{1in}
\includegraphics[width=5.4in]{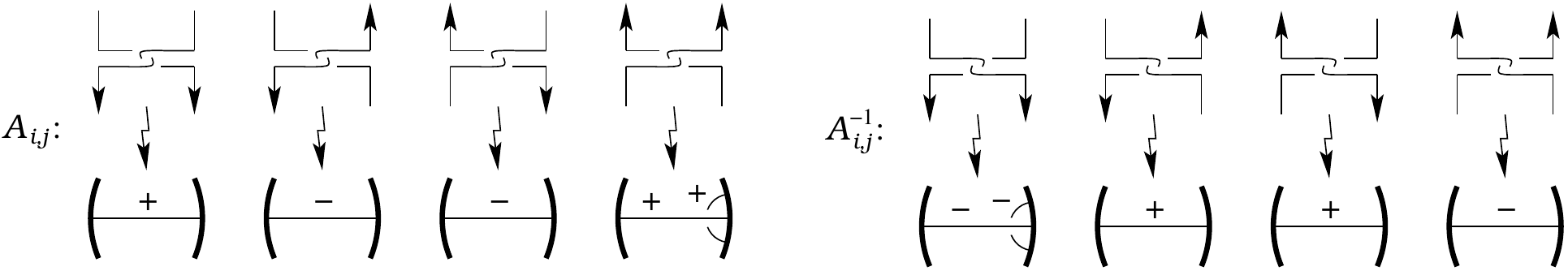}
$$
See Figure~\ref{fig:guatever} for the explanation of these rules.

The short-circuit moves of the first type translate into the moves $\mathrm{C}_1$. 
For the short-circuit moves of the second type it suffices to consider $\eps=+1$ if $r$ is odd and 
$\eps=-1$ if $r$ is even. These moves translate into the moves $\mathrm{C}_2$ on clasp diagrams.
 
Now, consider the Margalit-McCammond relations in the pure braid group.

The relation (MM1) translates into the move A. 

For the relation (MM2) there are $2^3=8$ different cases depending on the parities of $i$, $r$ and $s$.
In order to simplify the translation of (MM2), we consider the original relations 
$$A_{i,r}A_{r,s}A_{i,s}=A_{r,s}A_{i,s}A_{i,r}=A_{i,s}A_{i,r}A_{r,s}$$ 
if at least two of the indices $i$, $r$, $s$ are odd and the inverse relation 
$$A_{i,s}^{-1}A_{r,s}^{-1}A_{i,r}^{-1}=A_{i,r}^{-1}A_{i,s}^{-1}A_{r,s}^{-1}
=A_{r,s}^{-1}A_{i,r}^{-1}A_{i,s}^{-1}$$ 
if at least two of the indices $i$, $r$, $s$ are even.
The relations (MM2) translate into the following moves, that we denote by $\mathrm{M}_1$, 
$\mathrm{M}_2$ and $\mathrm{M}_3$:
$$\fig{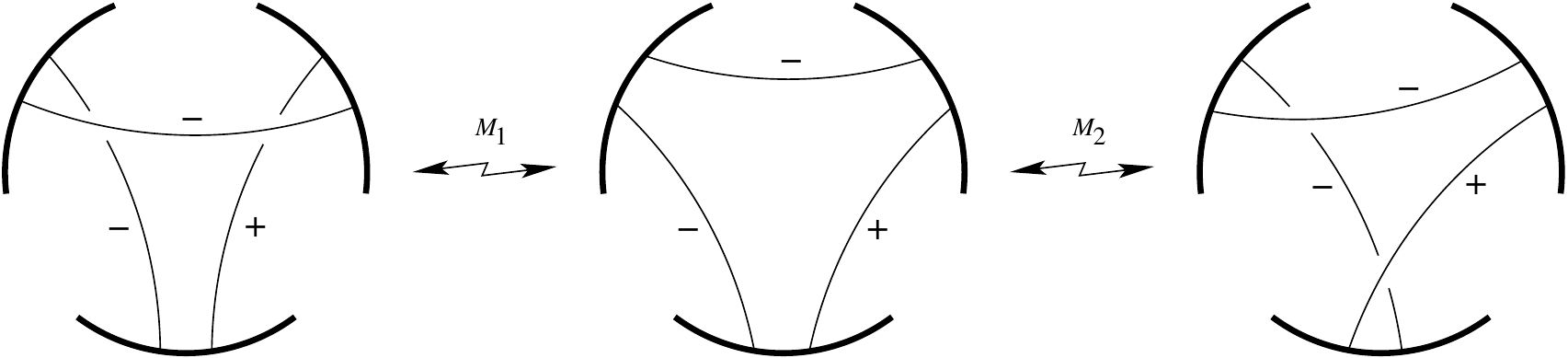}$$
if one of $i$, $r$, $s$ is even and
$$\fig{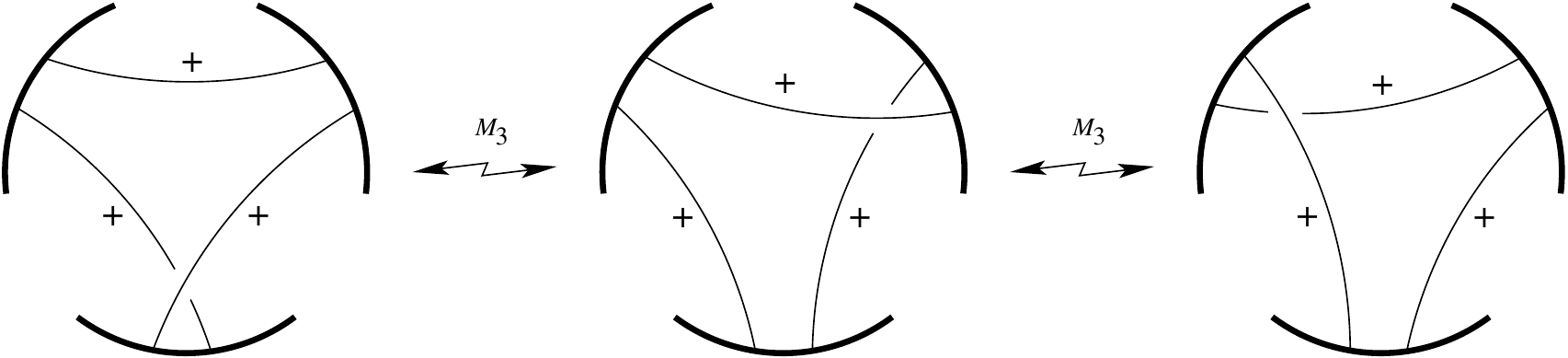}$$
if all indices are odd.
The inverse relations translate into the mirror moves $\mathrm{M}_1^*$, $\mathrm{M}_2^*$ 
if one of the indices is odd and into the mirror move $\mathrm{M}_3^*$ if all indices are even.

For the relation (MM3) there are $2^4=16$ different cases depending on the parities of 
$i$, $r$, $j$ and $s$.
In order to simplify the translation of (MM3), we consider the (equivalent) relations 
$$A_{i,j}^\eps\left( A_{i,r}A_{r,s}^{\delta \sigma} A_{i,r}^{-1}\right) = 
\left( A_{i,r}A_{r,s}^{\delta \sigma} A_{i,r}^{-1}\right) A_{i,j}^\eps$$ 
where $\eps=+1$ if $i$ and $j$ are of the same parity and $\eps=-1$ otherwise; 
$\delta=+1$ if $i$ and $r$ are of the same parity and $\delta=-1$ otherwise; 
and $\sigma=1$ if $s$ is odd and $\sigma=-1$ otherwise. 

If $i$ is odd and $r$ is even, these relations translate into 
$$\includegraphics[height=0.8in]{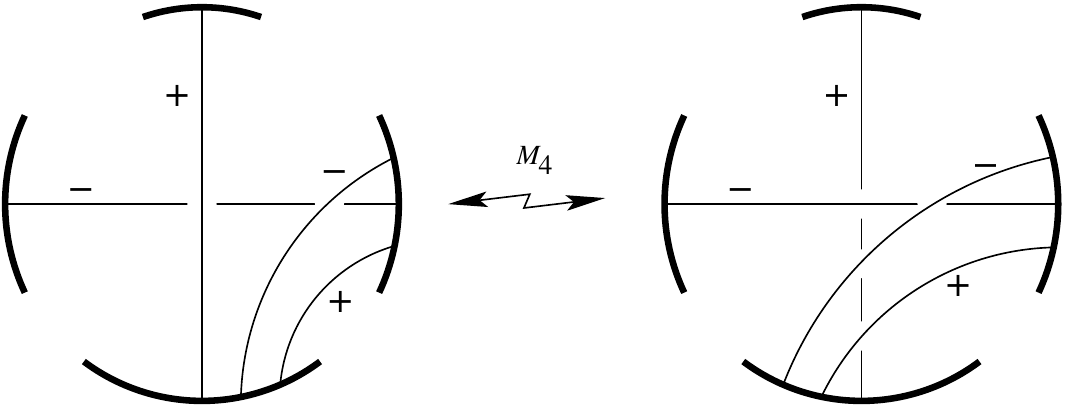}$$

If both $i$ and $r$ are odd, these relations translate into 
$$\includegraphics[height=0.8in]{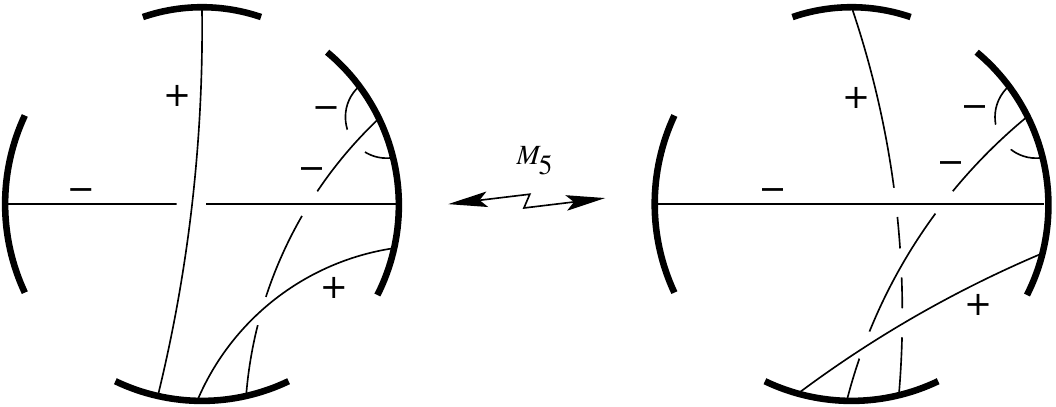}$$
when $s$ is even, and into 
$$\includegraphics[height=0.8in]{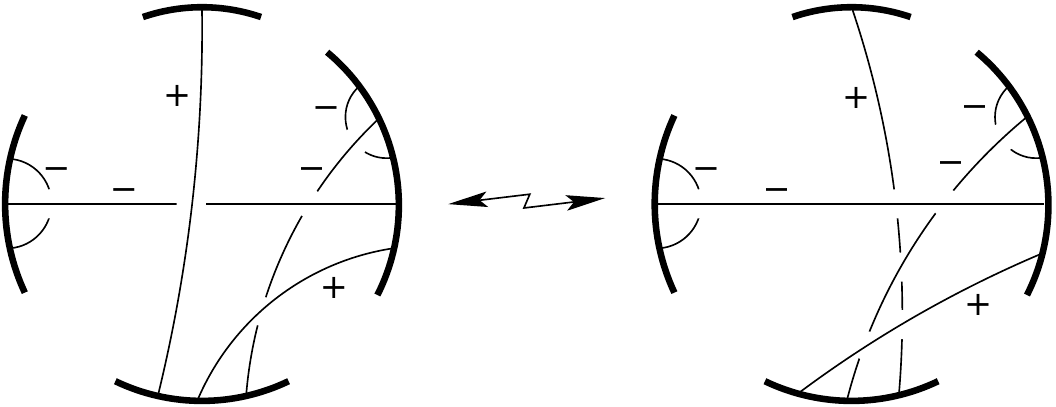}$$
when $s$ is odd.
Obviously, the last move follows from $\mathrm{M}_5$.

Finally, for even $i$ these relations translate into the corresponding mirror moves and we obtain

\begin{prop}
Any two clasp diagrams that gives rise to isotopic knots are related by a finite sequence of moves 
$\mathrm{A}$, $\mathrm{B}$, $\mathrm{C}_1$, $\mathrm{C}_2$, $\mathrm{M}_1$--$\mathrm{M}_5$ and 
$\mathrm{M}_1^*$--$\mathrm{M}_5^*$.  
\end{prop}

It remains to deduce $\mathrm{M}_1$--$\mathrm{M}_5$ and their mirrors from 
$\mathrm{A}$, $\mathrm{B}$, $\mathrm{C}_1$, $\mathrm{C}_2$, $\mathrm{C}_4$.
\medskip 
Let us start with $\mathrm{M}_1$ and $\mathrm{M}_2$:
\begin{equation}
\fig{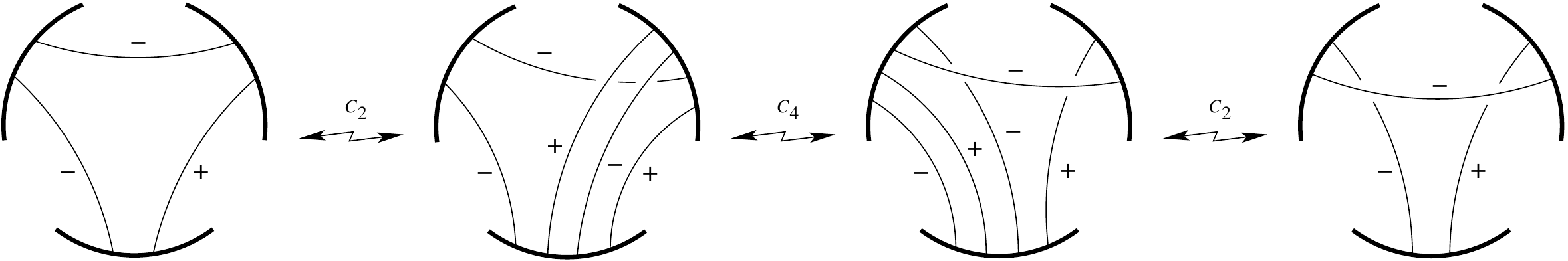}
\tag{$\mathrm{M}_1$}\label{eq:MM1_proof} 
\end{equation}

\begin{equation}
\fig{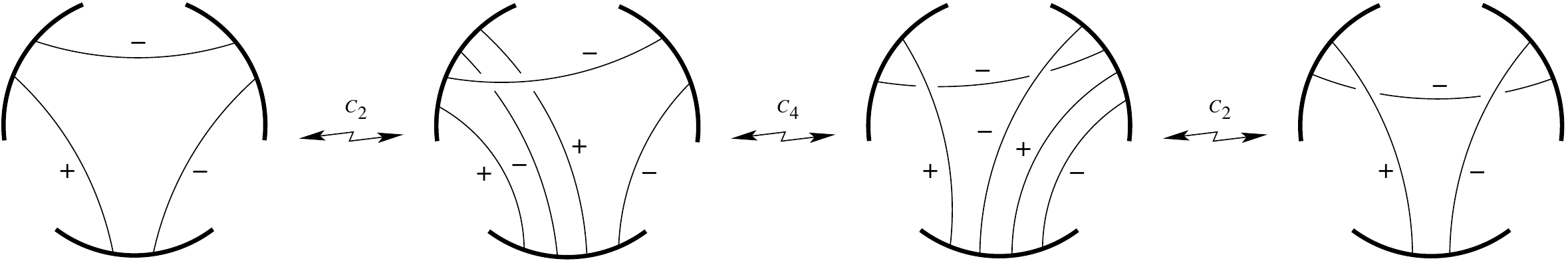}
\tag{$\mathrm{M}_2$}\label{eq:MM2_proof} 
\end{equation}
\\

The mirror moves $\mathrm{M}_1^*$ and $\mathrm{M}_2^*$ are also easy to obtain:

\begin{equation}
\fig{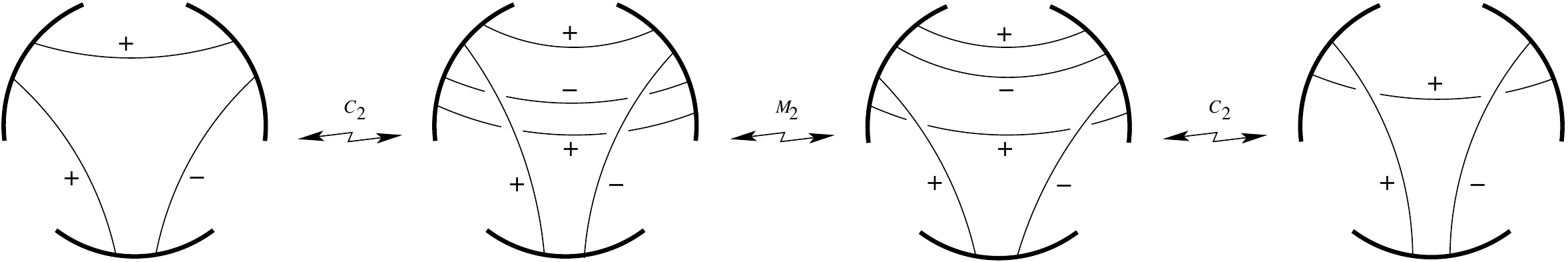}
\tag{$\mathrm{M}_1^*$}\label{eq:MM1star_proof}
\end{equation}
\\

\begin{equation}
\fig{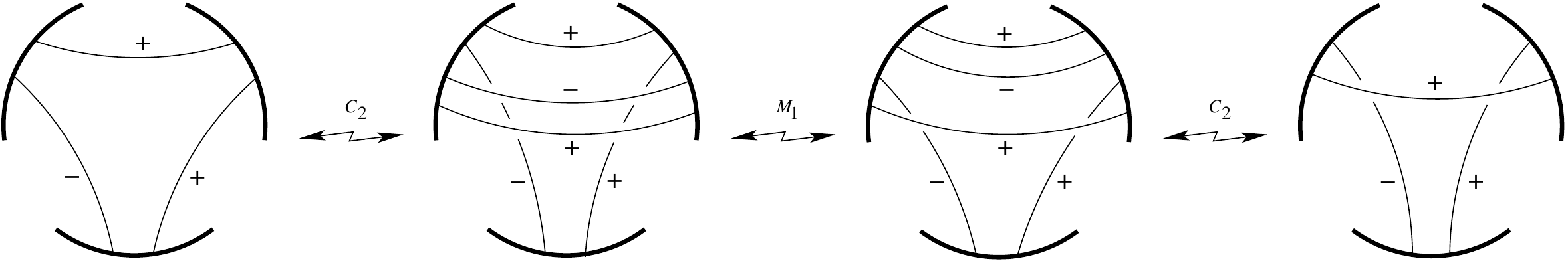}
\tag{$\mathrm{M}_2^*$}\label{eq:MM2star_proof} 
\end{equation}
\\

An auxiliary move $\mathrm{D}_1$ is needed to deal with $M_3$:
\begin{equation}
\fig{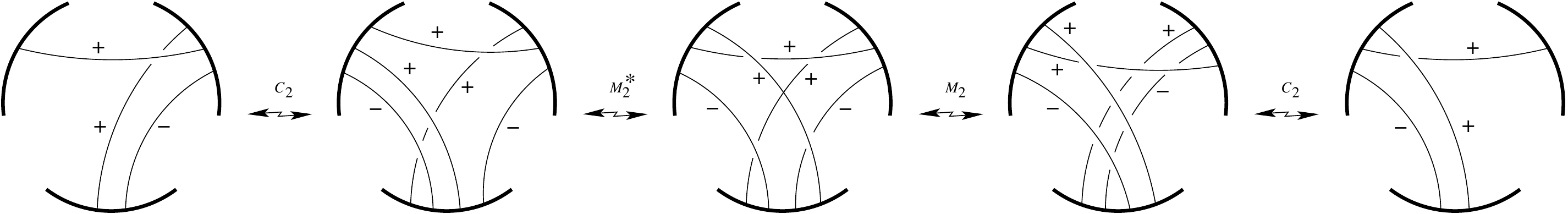}
\tag{$\mathrm{D}_1$}\label{eq:D1_proof} 
\end{equation}
\\

Note that the mirror move $\mathrm{D}_1^*$ also can be obtained from $\mathrm{C}_2$ and 
$\mathrm{M}_2$, $\mathrm{M}_2^*$ by mirroring this figure.
Now we are ready to deduce $\mathrm{M}_3$:

\begin{equation}
\fig{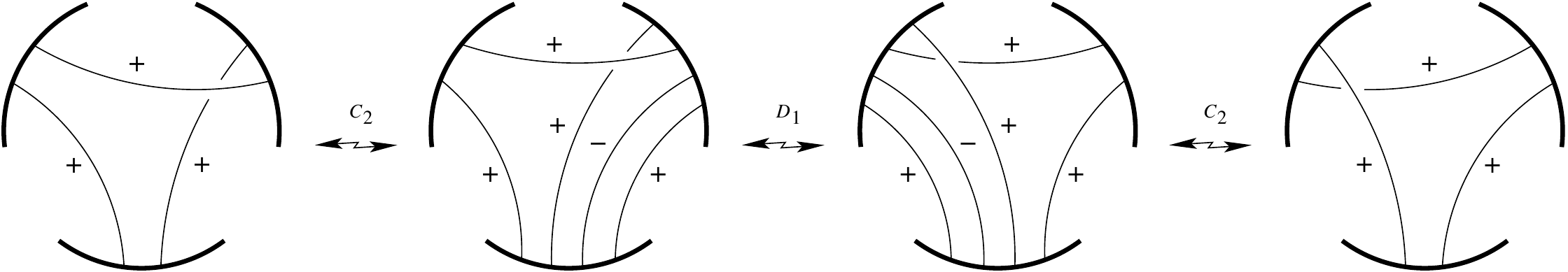}
\tag{$\mathrm{M}_3$}\label{eq:MM3_proof} 
\end{equation}

The mirror move $\mathrm{M}_3^*$ can be obtained from $\mathrm{C}_2$ and $\mathrm{D}_1^*$ 
by mirroring this figure.
Also, $\mathrm{M}_4$ can be obtained from $\mathrm{C}_4$ and $\mathrm{D}_1^*$: 

\begin{equation}
\fig{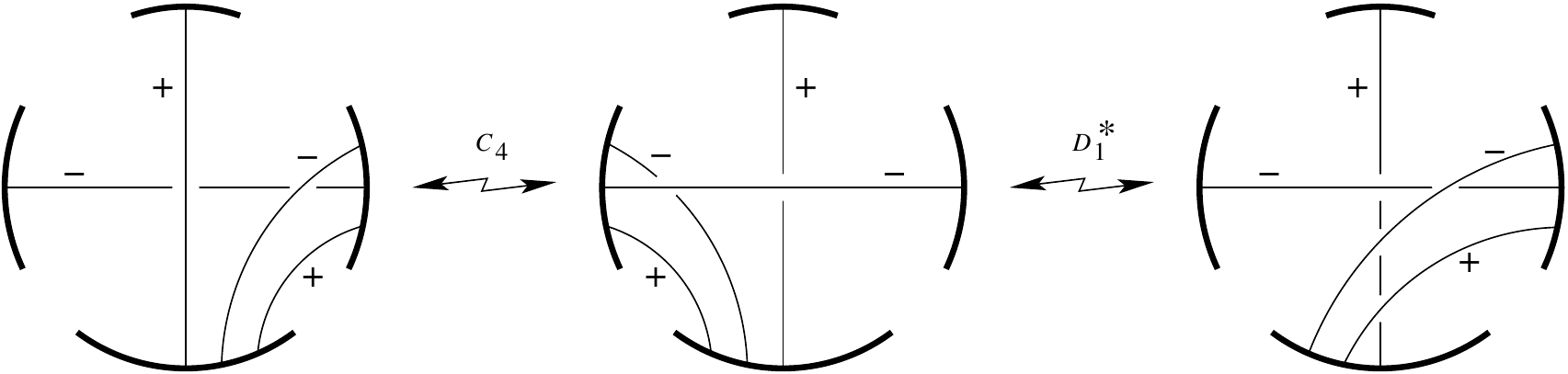}
\tag{$\mathrm{M}_4$}\label{eq:MM4_proof} 
\end{equation}

Now the mirror move $\mathrm{C}_4^*$ can be obtained from $\mathrm{C}_2$, $\mathrm{D}_1$ 
and $\mathrm{M}_4$:

\begin{equation}
\fig{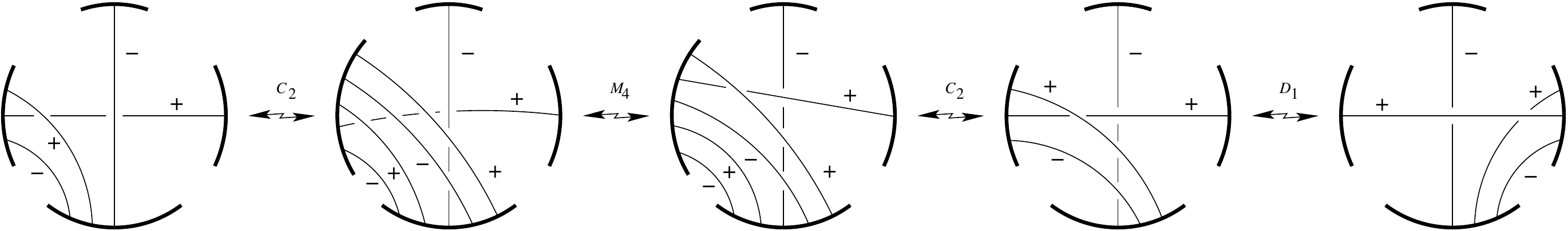}
\tag{$\mathrm{C}_4^*$}\label{eq:C4star_proof}
\end{equation}
\\

Thus $\mathrm{M}_4^*$ can be obtained from $\mathrm{C}_4^*$ and $\mathrm{D}_1$
mirroring the figure for $\mathrm{M}_4$ above.
Two additional auxiliary moves $\mathrm{D}_2$ and $\mathrm{D}_3$ are required to deal 
with $\mathrm{M}_5$:

\begin{equation}
\fig{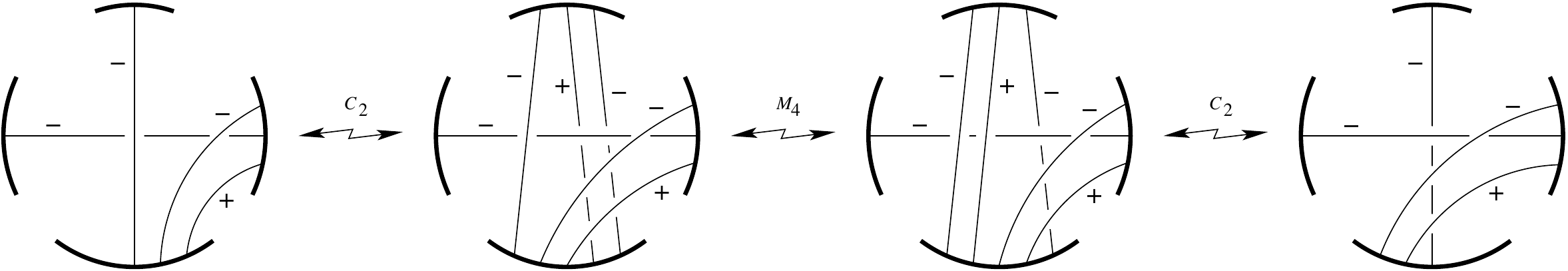}
\tag{$\mathrm{D}_2$}\label{eq:D2_proof} 
\end{equation}

\begin{equation}
\fig{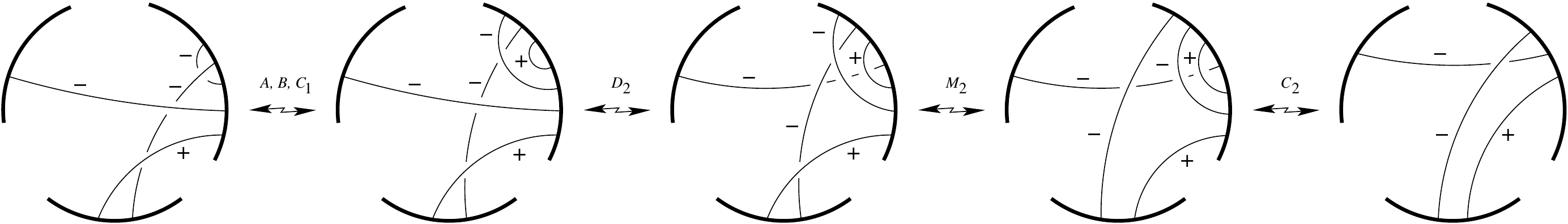}
\tag{$\mathrm{D}_3$}\label{eq:D3_proof} 
\end{equation}

We can also obtain the mirror moves $\mathrm{D}_2^*$ and then $\mathrm{D}_3^*$ 
mirroring these figures. Finally we can deduce $\mathrm{M}_5$:

\begin{equation}
\fig{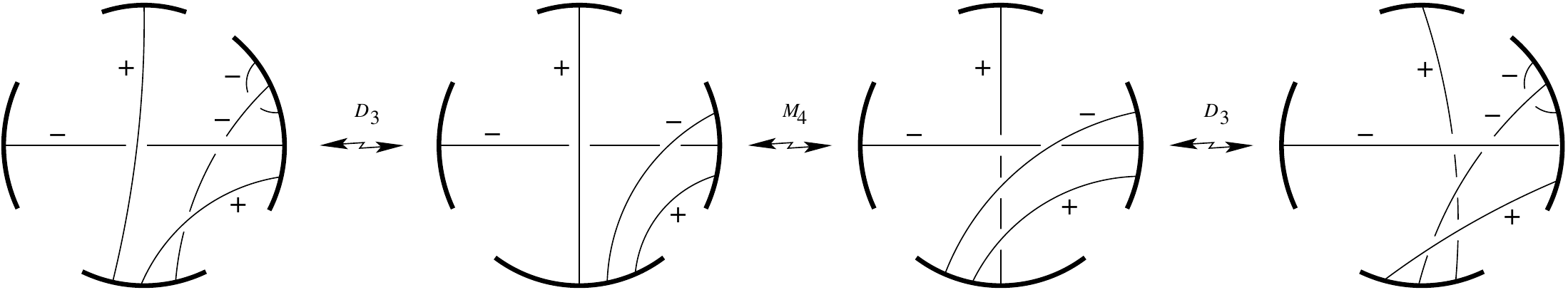}
\tag{$\mathrm{M}_5$}\label{eq:MM5_proof} 
\end{equation}

Mirroring this figure gives the mirror move $\mathrm{M}_5^*$. 
This finishes the proof of Theorem~\ref{thm:moves}.

\end{document}